\documentclass[10pt]{article}

\usepackage[english]{babel}

\usepackage[letterpaper,top=2cm,bottom=2cm,left=3cm,right=3cm,marginparwidth=1.75cm]{geometry}

\usepackage{amsmath}
\usepackage{amssymb}
\usepackage{graphicx}
\usepackage[colorlinks=true, allcolors=blue]{hyperref}
\usepackage[shortlabels]{enumitem}

\newcommand{\dive}{\operatorname{div}}
\usepackage{amsthm}
\newtheorem{theorem}{Theorem}[section]
\newtheorem{lemma}[theorem]{Lemma}
\newtheorem{proposition}[theorem]{Proposition}
\newtheorem{corollary}[theorem]{Corollary}

\newtheorem{remark}[theorem]{Remark}
\newtheorem{definition}{Definition}[section]

\title{Recovery of an Anisotropic Conductivity from the Neumann-to-Dirichlet Map in a Semilinear Elliptic Equation}

\author{
Elena Beretta\textsuperscript{1,*} \and
Elisa Francini\textsuperscript{2} \and
Dario Pierotti\textsuperscript{3} \and
Eva Sincich\textsuperscript{4}
}

\date{}

\begin{document}

\maketitle

\footnotetext[1]{NYU Abu Dhabi, UAE. Corresponding author: \texttt{eb147@nyu.edu}.}
\footnotetext[2]{Università di Firenze, Italy.}
\footnotetext[3]{Politecnico di Milano, Italy.}
\footnotetext[4]{Università di Trieste, Italy.}

\begin{abstract}
We study the inverse boundary value problem of detecting a non-uniform conductivity motivated by pacing-guided ablation in cardiac electrophysiology. At the stationary level, the transmembrane potential $u$ in a region \(\Omega\subset\mathbb{R}^3\) of cardiac tissue satisfies
\[
-\nabla\!\cdot(\gamma\nabla u)+\alpha u^3=0 \quad \text{in }\Omega,\qquad 
\gamma\nabla u\cdot\nu=g \quad \text{on }\partial\Omega,
\]
where $\gamma$ is an anisotropic conductivity tensor and $\alpha$ a nonlinear ionic response coefficient. The Neumann data $g$ represent pacing currents, and the boundary values $u|_{\partial\Omega}$ correspond to invasive voltage measurements. Ischemic regions are modeled by a subdomain $D\subset\Omega$ where $\gamma$ is piecewise constant.
We address the inverse problem of determining $\gamma$ from the Neumann-to-Dirichlet (NtD) map, assuming that $\alpha$ and $D$ are known. To our knowledge, uniqueness in the case of NtD data with anisotropic conductivities in this nonlinear setting has not been analyzed in previous work. Using a first-order linearization around a nontrivial pacing current, we prove uniqueness for $\gamma$.

\end{abstract}

\section{Introduction}

In this paper, we analyze a three-dimensional mathematical model stemming from cardiac electrophysiology, where the goal is to detect the conduction properties of ischemic or scarred regions in the heart from boundary measurements. Ischemia and post‑infarct scar alter the ionic properties and structural organization of the tissue, leading to impaired propagation of the electrical signal, conduction slowing and block, and potentially ventricular arrhythmia and fibrillation, ultimately causing death. In order to prevent this, electrophysiologists perform ablation procedures using pacing: a catheter delivers controlled electrical stimuli at selected sites of the endocardium, and the resulting electrograms are recorded to probe local conduction properties, \cite{pagani2021data,pagani2021computational}.

 Regions exhibiting conduction block or slowed propagation can thus be identified as arrhythmogenic substrates and targeted for ablation. Mathematically, this corresponds to prescribing boundary currents and observing the induced potentials, precisely the input–output setting which we can interpret within an electrical impedance tomography (EIT) framework to reconstruct the effective conductivity.
 At the scale of continuum models such as the monodomain equation, these conduction abnormalities are described through effective parameters that summarize the underlying ionic and structural changes. In the simplified setting considered here, we encode the effect of an ischemic or scarred region as a spatial variation of an effective conductivity tensor $\gamma$, while the nonlinear ionic response coefficient $\alpha$ will be assumed known. Thus, in what follows, a change in conductivity should be understood as an effective macroscopic description of altered propagation, not as a direct statement about microscopic ionic conductivities.
Here, the mathematical model we use to describe the electrical activity of the heart is the stationary version of the monodomain model. The full monodomain model is a nonlinear reaction–diffusion parabolic equation coupled with a system of ionic ODEs that accurately describes the spatio–temporal evolution of action potentials, \cite{quarteroni2019mathematical,trayanova2024computational,trayanova2024up,pagani2021data,pagani2021computational}. 
Physiologically, when the tissue is subjected to repeated external stimulation (pacing), transient oscillations subside and the potential distribution approaches a quasi-stationary state. In this regime, the full monodomain system can be reduced to an effective diffusion problem for the transmembrane potential, in which anisotropic conduction is encoded by the tensor $\gamma$ and the nonlinear ionic response is summarized by the coefficient $\alpha$. Freezing the time variable thus provides a meaningful approximation of the equilibrium reached under steady pacing, leading to the semilinear elliptic problem below, which isolates the effective conductive properties of the tissue \cite{franzone2014mathematical}.


The mathematical model we want to study reads as follows
\begin{equation}\label{nonlineareq}
    \left\{
    \begin{aligned}
      -\dive \left(\gamma \nabla u\right)+\alpha u^3=0  &  \mbox{ in }\Omega\\
      \gamma \nabla u\cdot \nu=g    & \mbox{ on }\partial\Omega
    \end{aligned}\right.
\end{equation}
with the following interpretations in electrophysiology:\\ $\Omega \subset \mathbb{R}^3$ denotes the portion of cardiac tissue under investigation (for instance, a patch of atrial or ventricular myocardium),
 $u(x)$ represents the \emph{transmembrane potential}, i.e., the difference between intracellular and extracellular electric potential at a point $x\in\Omega$. In the stationary setting, $u$ describes the equilibrium potential distribution under steady pacing,  $\gamma$ is the \emph{anisotropic conductivity tensor} of the tissue, modeling the preferential propagation of current along myocardial fibers versus across them,  $\alpha$ is a \emph{nonlinear ionic response coefficient}, summarizing the contribution of ionic currents through the membrane. The cubic term $\alpha u^3$ is a simplified nonlinear current--voltage relation encoding excitability effects.
The diffusion term $-\nabla\!\cdot(\gamma \nabla u)$ describes the propagation of the transmembrane potential through the tissue due to anisotropic conduction.
The Neumann boundary condition $\gamma \nabla u \cdot \nu = g$ on $\partial\Omega$ prescribes the applied current density at the boundary, where $\nu$ denotes the unit outward normal. In practice, $g$ corresponds to the stimulus delivered by pacing electrodes in invasive electrophysiological procedures. The associated Neumann-to-Dirichlet map, which sends $g$ to the resulting trace $u|_{\partial\Omega}$, mathematically models the measurements obtained in invasive EIT.

In our model, the presence of an ischemic or scarred region $D \subset \Omega$ is represented by a piecewise constant effective conductivity tensor
\begin{equation}\label{gammaint}
\gamma(x)=\gamma_0\chi_{\Omega\setminus D}(x)+\gamma_1\chi_{D}(x),
\end{equation}
where $\gamma_i$ are constant symmetric positive definite matrices for $i=0,1$. In other words, we model the arrhythmogenic substrate as a subdomain in which the effective conduction properties differ from the surrounding healthy tissue. The ionic response coefficient $\alpha$ is assumed to be a bounded and positive function.

Although this is a simplified stationary model compared to the full time-dependent description, it is mathematically significant, as it already exhibits the main analytical difficulties of the nonlinear inverse boundary value problem we want to address.

In practice, the geometry of ischemic or scarred regions is identified in advance by imaging, see for example \cite{paddock2021clinical}, whereas their effective conduction properties cannot be measured directly and must be inferred from pacing responses. We therefore assume that the inclusion $D$ is known and focus on recovering the anisotropic conductivity from knowledge of the NtoD map.

We recall that the identification of $D$ itself from boundary measurements has been addressed in related works. In \cite{beretta2016asymptotic,beretta2017reconstruction}, some of the present authors studied the recovery of a small inclusion $D$ with constant conductivity $k$, with $0 < k \ll 1$, embedded in a homogeneous background of unit conductivity, while in \cite{beretta2018detection} the case of inclusions of arbitrary size was analyzed.
In  principle, the recovery of $D$ is also possible in the case of anisotropic conductivities treated here as pointed out in Section 4.5 and will be topic of future work.


A nowadays well-established technique based on linearization was introduced in the 1990s for recovering the potential in semilinear equations from knowledge of the Dirichlet-to-Neumann map. The central idea is to exploit the smooth dependence of the DtoN map on perturbations of the boundary data and to analyze its first-order linearization \cite{isakov1994global}. Since then, this approach has been extended to more general potentials through higher-order linearization around homogeneous Dirichlet boundary data (cf.\cite{lassas2021inverse}, the review paper  \cite{lassas2025introduction} and references therein). More recently, Harrach and Lin \cite{harrach2023simultaneous} showed that this method allows the simultaneous recovery of piecewise analytic isotropic conductivities and potentials of the form 
$a(x,u)$, which are piecewise analytic in 
$x$ and analytic in 
$u$, from the nonlinear DtN map under suitable smoothness assumptions on the boundary data. 

To the best of our knowledge, no analogous results have been obtained in the case of Neumann-to-Dirichlet (NtD) data corresponding to anisotropic conductivities.
In this paper, we obtain a result of uniqueness of the tensor $\gamma$ of the form \eqref{gammaint} (assuming $D$ and $\alpha$  known)
from the NtoD map via a first order linearization around a nontrivial datum $g_0$. 

This corresponds to a clinically relevant pacing current that produces a nontrivial potential
$u_0\neq 0$ solution to \eqref{nonlineareq} with $g=g_0$, and the linearized equation
\begin{equation}\label{linearizedeq}
-\nabla\!\cdot(\gamma\nabla v)+3\alpha u_0^2 v=0
\end{equation}
already contains both the conductivity $\gamma$ and the nonlinear coefficient
$\alpha$ at first order. 
The main obstruction in determining both $\gamma$ and $\alpha$ in the linearized equation
\eqref{linearizedeq} is the dependence of the potential $q:=3\alpha u_0$ on $\gamma$.
The linear problem has been studied in \cite{Al-dH-G} in the case where the potential $q=0$.

Open issues of interest that we plan
to investigate in future work are stability, reconstruction algorithms, and simultaneous
recovery of $\gamma$ and $\alpha$. We note that, in the isotropic case, a simultaneous
reconstruction of $\gamma$ and $\alpha$ is in fact possible by using higher-order
linearizations of the NtoD map around the trivial solution, as shown for the DtoN map in~\cite{harrach2023simultaneous}. However, this approach is not meaningful in the
applications we have in mind, where nontrivial stimuli are used during the treatment of
patients, and the trivial background state is therefore not physically relevant.


The structure of the paper is as follows: in Section 2, we investigate the forward problem and determine some crucial estimates for the solution to \eqref{nonlineareq}.
 In Section 3 we establish some key estimates for the solution $u_0$ of \eqref{nonlineareq} with nontrivial datum $g_0\geq 0$ that allow us to compute the Fr\'echet derivative of the NtoD map. In Section 4, we prove that the NtoD map uniquely determines both $\gamma$. Finally, Section 5 contains some extensions and remarks.


\section{Well posedness of the direct problem}

\begin{definition}\label{regbordo}
Let $\Omega$ be a bounded domain in $\mathbb{R}^3$. We say that $\partial \Omega$ is of class $C^{1,1}$ if for any $P\in \partial \Omega$ there exists a rigid transformation of coordinates under which we have $P=0$ and 
\begin{equation*}
    \Omega \cap B_{r_0} =\{x\in B_{r_0} | x_n > \varphi(x')  \}
\end{equation*}
where $\varphi$ is a $C^{1,1}$ function on $B'_{r_0}$ satisfying 
\begin{equation*}
\varphi(0)=|\nabla \varphi(0)|=0. 
\end{equation*}
\end{definition}

Let $D\subset\subset \Omega$ such that
both $\partial \Omega$ and $\partial D$ are of class $C^{1,1}$.

Let $\gamma\in L^{\infty}(\Omega,\textrm{Sym}_3)$,
where we denote with {\textrm{Sym}$_3$} the class of $3\times 3$ symmetric real valued matrices. 
In particular let us assume that $\gamma$ is 
piecewise constant,  that is 
\begin{equation}\label{gamma}
\gamma(x)=\gamma_0\chi_{\Omega\setminus D}+\gamma_1\chi_{D}
\end{equation}
where $\gamma_i, \ i=0,1$ are constant matrices in {\textrm{Sym}$_3$}.
Let also assume that there is a positive number $\lambda_0\leq 1$ such that $\gamma_i$ satisfies the ellipticity condition 
\begin{equation}\label{ell}
    \lambda_0|\xi|^2\leq \gamma_i \xi \cdot \xi\leq \lambda_0^{-1}|\xi|^2 \quad\mbox{for every}\quad \xi\in\mathbb{R}^3\quad\mbox{and for}\quad i=0,1.
\end{equation}
Let $\alpha(x)\in L^\infty(\Omega)$ such that, for some  $0<\alpha_0<1$, 
\begin{equation}\label{hpalpha}
    \alpha_0\leq\alpha(x)\leq \alpha_0^{-1} \mbox{ for }x\in \Omega.
\end{equation}

Given $g\in H^{-1/2}\left(\partial\Omega\right)$, let us consider the Neumann boundary value problem
\begin{equation}\label{PN}
    \left\{
    \begin{aligned}
      -\dive \left(\gamma \nabla u\right)+\alpha u^3=0  &  \mbox{ in }\Omega\\
      \gamma \nabla u\cdot \nu=g    & \mbox{ on }\partial\Omega.
    \end{aligned}\right.
\end{equation}
Let us first show that problem \eqref{PN} is well posed.
Since we state the result for the weak form of the problem, it is sufficient to assume that $\Omega$ has Lipschitz boundary.
\begin{proposition}\label{wellposedness}
Let $\Omega$ be a bounded Lipschitz domain and let $\gamma$ be given by \eqref{gamma} and satisfying \eqref{ell}, and $\alpha\in L^{\infty}(\Omega)$ satisfying  \eqref{hpalpha}.
Then for every $g\in H^{-1/2}\left(\partial \Omega\right)$ there is a unique solution to the Neumann boundary value problem \eqref{PN}.
Moreover, 
\begin{equation}\label{wellpos}
    \|u\|_{H^1(\Omega)}\leq \frac{1}{\lambda_0}\left(\|g\|_{H^{-1/2}(\partial\Omega)}+\left(\frac{|\Omega|}{\alpha_0}\right)^{1/3}\|g\|^{1/3}_{H^{-1/2}(\partial\Omega)}\right).
\end{equation}
\end{proposition}
Here and in the following, we will denote by
\begin{equation}
   \label{solu}
   u(\cdot,g)
\end{equation}
the unique solution of problem \eqref{PN}

\begin{proof}
A weak solution of problem \eqref{PN} is a function $u\in H^1(\Omega)$ such that, for every test function $\phi\in H^1(\Omega)$ 
\begin{equation}\label{weak}
    \int_\Omega \gamma\nabla u \cdot\nabla \phi+\int_\Omega \alpha u^3\phi=<g,\phi>_{\partial\Omega}
\end{equation}
where $<\cdot,\cdot>_{\partial\Omega}$ is the duality pairing in $H^{-1/2}(\partial\Omega)$.

Let us define the energy functional 
$E: H^1(\Omega)\rightarrow\mathbb{R}$ given by
    \begin{equation*}
    E(u)=\frac{1}{2}\int_\Omega \gamma\nabla u\cdot\nabla u +\frac{1}{4}\int_\Omega \alpha u^4-<g,u>_{\partial\Omega}.
\end{equation*}
The weak formulation \eqref{weak} corresponds to finding critical point of $E(u)$. Let us show that $E$ has a minimizer in $H^1(\Omega)$ (which will be a critical point and, hence a weak solution of problem \eqref{PN}) using the direct method of calculus of variations, i.e. showing that 
\begin{enumerate}[a)]
\item $E(u)$ is coercive, i.e.
\begin{equation*}
    E(u)\rightarrow+\infty \mbox{ as }\|u\|_{H^1(\Omega)}\rightarrow +\infty;
\end{equation*}
    \item $E(u)$ is lower semicontinuous;
    \item A minimizing sequence $\{\tilde {u}_n\}$ converges to some $u\in H^1(\Omega)$;
    \item The limit $u$ is indeed a minimizer for $E(u)$, hence a weak solution for \eqref{PN}.
\end{enumerate}
Let us start with step a). By assumptions \eqref{ell} and \eqref{hpalpha} we have
\begin{equation*}
    E(u)\geq \frac{\lambda_0}{2}\int_\Omega|\nabla u|^2+\frac{\alpha_0}{4}\int_\Omega u^4-<g,u>_{\partial\Omega}
\end{equation*}
Now, by Schwarz inequality it follows,
\begin{equation}\label{Jen}
    \int_\Omega u^4\geq \frac{1}{|\Omega|}\left(\int_\Omega u^2\right)^2=\frac{1}{|\Omega|}\|u\|^4_{L^2(\Omega)},
\end{equation}
and
\begin{equation*}
   \left| <g,u>_{\partial\Omega}\right|\leq \|g\|_{H^{-1/2}(\partial\Omega)}\|u\|_{H^{1/2}(\partial\Omega)}\leq \|g\|_{H^{-1/2}(\partial\Omega)}\|u\|_{H^1(\Omega)},
\end{equation*}
hence, by elementary calculations,
\begin{align*}
    E(u)&\geq \frac{\lambda_0}{2}\|\nabla u\|^2_{L^2(\Omega)}+\frac{\alpha_0}{4|\Omega|}\|u\|^4_{L^2(\Omega)}-\|g\|_{H^{-1/2}(\partial\Omega)}\|u\|_{H^1(\Omega)}\nonumber\\
    &\geq \frac{\lambda_0}{2}\|u\|^2_{H^1(\Omega)}+\frac{\alpha_0}{4|\Omega|}\|u\|^4_{L^2(\Omega)}-\frac{\lambda_0}{2}\| u\|^2_{L^2(\Omega)}-\|g\|_{H^{-1/2}(\partial\Omega)}\|u\|_{H^1(\Omega)}\nonumber\\
    &\geq \frac{\lambda_0}{2}\| u\|^2_{H^1(\Omega)}-\frac{\lambda_0^2|\Omega|}{4\alpha_0}-\|g\|_{H^{-1/2}(\partial\Omega)}\|u\|_{H^1(\Omega)}
\end{align*}
which proves the coercivity of $E(u)$ since
\begin{equation*}
    \frac{E(u)}{\| u\|_{H^1(\Omega)}}\geq \frac{\lambda_0}{2}\| u\|_{H^1(\Omega)}-\frac{\lambda_0^2|\Omega|}{4\alpha_0}\frac{1}{\| u\|_{H^1(\Omega)}}-\|g\|_{H^{-1/2}(\partial\Omega)}\rightarrow+\infty
\end{equation*}
if 
\begin{equation*}
   \|u\|_{H^1}\rightarrow +\infty. 
\end{equation*}
Let us now show that $E(u)$ is weakly lower continuous. Let $u_n$ be a sequence of functions  such that $u_n\rightharpoonup u$ in $H^1(\Omega)$. This implies that 
\begin{equation*}
   \|u\|_{H^1(\Omega)}\leq \liminf_{n\to+\infty} \|u_n\|_{H^1(\Omega)}
\end{equation*}
and, by a compactness argument, by possibly considering a suitable subsequence,
\begin{equation*}
    \|u_n\|_{L^4(\Omega)}\rightarrow\|u\|_{L^4(\Omega)}.
\end{equation*}
Moreover, since the trace operator $Tr: H^1(\Omega)\rightarrow H^{1/2}(\partial\Omega)$ is linear and continuous, we also have that
\begin{equation*}
    Tr(u_n)\rightarrow Tr(u)\mbox{ in }H^{1/2}(\partial\Omega),
    \end{equation*}
hence
\begin{equation}\label{semicont}
    E(u)\leq \liminf_{n\to+\infty} E(u_n).
\end{equation}

Now, by coercivity of $E$, there exist a positive constant $M$ such that 
\begin{equation*}
    E(u)\geq -M\mbox{ for every }u\in H^1(\Omega),
\end{equation*}
hence 
\begin{equation*}
   \tilde{M}=\inf_{u\in H^1(\partial\Omega)} E(u)\geq -M.
\end{equation*}
Given a minimizing sequence $\tilde{u}_n$ such that
\begin{equation*}
   E(\tilde{u}_n)\rightarrow\tilde{M} \mbox{ for }n\to+\infty,
\end{equation*}
since $E(\tilde{u}_n)$ is bounded, by coercivity, the sequence $\tilde{u}_n$ is bounded in $H^1(\Omega)$ norm and, up to subsequences 
\begin{equation*}
    \tilde{u}_n\rightharpoonup u \mbox{ in }H^1(\Omega)
\end{equation*}
By \eqref{semicont}, $E(u)$ is lower semicontinuous and hence, we have that
\begin{equation*}
    E(u)\leq \liminf_{n\to+\infty} E(\tilde{u}_n)=  \lim_{n\to+\infty} E(\tilde{u}_n)=\tilde{M}
\end{equation*}
so that $u$ is a minimum point for $E(u)$ and, hence, it satisfies equation \eqref{weak}.

Let us now show that this solution is unique. Let $u_1$ and $u_2$ be two solutions of \eqref{PN} and let $w=u_1-u_2$. Then $w\in H^1(\Omega)$ solves
\begin{equation*}
    \left\{
    \begin{aligned}
      -\dive \left(\gamma \nabla w\right)+\alpha w (u_1^2+u_1u_2+u_2^2)=0  &  \mbox{ in }\Omega\\
      \gamma \nabla w\cdot \nu=0    & \mbox{ on }\partial\Omega
    \end{aligned}\right.
\end{equation*}
which gives
\begin{equation*}
    \int_\Omega \gamma\nabla w\cdot\nabla w+\int_\Omega \alpha q w^2=0,
\end{equation*}
where 
\begin{equation*}
    q= (u_1^2+u_1u_2+u_2^2)\geq 0.
\end{equation*}
By \eqref{ell}, 
\begin{equation*}
    \lambda_0\int_\Omega |\nabla w|^2+\alpha_0\int_\Omega q w^2\leq 0,
\end{equation*}
from which we get
\begin{equation*}
    \int_\Omega |\nabla w|^2=0 \mbox{ and } \int_\Omega q w^2=0,
\end{equation*}
Which finally give $w=0$ almost everywhere in $\Omega$.

Let us choose $\phi=u$ in \eqref{weak} and, by \eqref{ell},\eqref{hpalpha}  we have
\begin{equation*}
    \lambda_0\|\nabla u\|_{L^2(\Omega)}^2+\frac{\alpha_0}{|\Omega|}\int_\Omega u^4\leq \|g\|_{H^{-1/2}(\partial\Omega)}\|u\|_{H^1(\Omega)}
\end{equation*}
which gives
\begin{equation*}
  \|\nabla u\|_{L^2(\Omega)}^2\leq \frac{1}{\lambda_0} \|g\|_{H^{-1/2}(\partial\Omega)}\|u\|_{H^1(\Omega)} 
\end{equation*}
and, by \eqref{Jen}
\begin{equation*}
    \|u\|_{L^2(\Omega)}^4\leq \frac{|\Omega|}{\alpha_0}\|g\|_{H^{-1/2}(\partial\Omega)}\|u\|_{H^1(\Omega)}
\end{equation*}
so that 
\begin{equation*}
  \|u\|_{H^1(\Omega)}^2\leq \frac{1}{\lambda_0} \|g\|_{H^{-1/2}(\partial\Omega)}\|u\|_{H^1(\Omega)} +
  \left(\frac{|\Omega|}{\alpha_0}\|g\|_{H^{-1/2}(\partial\Omega)}\right)^{1/2}\|u\|_{H^1(\Omega)}^{1/2}
\end{equation*}
that can be written as
\begin{equation}\label{st2}
  \|u\|_{H^1(\Omega)}^{1/2})\left(\|u\|_{H^1(\Omega)}-\frac{1}{\lambda_0} \|g\|_{H^{-1/2}(\partial\Omega)}\right)\leq
  \left(\frac{|\Omega|}{\alpha_0}\right)^{1/2}\|g\|_{H^{-1/2}(\partial\Omega)}^{1/2}.
\end{equation}
From \eqref{st2} we have that either
\begin{equation*}
\|u\|_{H^1(\Omega)}\leq\frac{1}{\lambda_0} \|g\|_{H^{-1/2}(\partial\Omega)},
\end{equation*}
otherwise
\begin{equation*}
 \left(\|u\|_{H^1(\Omega)}-\frac{1}{\lambda_0} \|g\|_{H^{-1/2}(\partial\Omega)}\right)^{3/2}\leq\|u\|_{H^1(\Omega)}^{1/2}\left(\|u\|_{H^1(\Omega)}-\frac{1}{\lambda_0} \|g\|_{H^{-1/2}(\partial\Omega)}\right)\leq
\left(\frac{|\Omega|}{\alpha_0}\right)^{1/2}\|g\|_{H^{-1/2}(\partial\Omega)}^{1/2}.
\end{equation*}
In either case we have \eqref{wellpos}.
\end{proof}
\begin{definition}[Neumann-to-Dirichlet map]\label{def:NDmap}
The Neumann to Dirichlet map (NtoD) corresponding to $\gamma$ in
\eqref{PN},
\[
\mathcal{N}_{NL}^{\gamma,\alpha} \colon H^{-1/2}(\partial\Omega) \longrightarrow H^{1/2}(\partial\Omega),
\]
is the map $g\in H^{-1/2}(\partial\Omega) \rightarrow u_{|{\partial\Omega}}\in H^{1/2}(\partial\Omega)$ where $u_{|{\partial\Omega}}$ is the trace of the unique weak solution $u\in H^1(\Omega)$ of the Neumann problem \eqref{PN} .
In other words, by using notation \eqref{solu}
\begin{equation*}
\mathcal{N}_{NL}^{\gamma,\alpha} (g)=u(\cdot,g)_{|{\partial\Omega}} 
\end{equation*}
\end{definition}


In order to discuss the properties of the Neumann-to-Dirichlet map $\mathcal{N}_{NL}^{\gamma,\alpha}$  and to address the problem of uniqueness in the inverse problem (see the following sections), it is necessary to deal with more regular solutions; in particular, $L^{\infty}$ estimates of the solutions will play a crucial role. As we show below, this can be achieved by assuming more regularity of the domains $\Omega$ and $D$ and on the Neumann datum $g$.

Preliminarily, we need to construct a regular function that meets the Neumann boundary condition.

\begin{lemma}\label{probw}
Let $\Omega$, $D$, be defined as in the problem \eqref{PN}. Suppose further that $\partial\Omega$ and $\partial D$ are of class $C^{1,1}$ and that $g\in W^{1-\frac{1}{p}}_p (\partial\Omega)$ with $p>3$.
Then, there exists a function $w$ defined in
$\Omega$ (and vanishing on $D$), satisfying the boundary condition
\begin{equation}\label{bdw}
    \gamma \nabla w\cdot\mathbf{\nu}=g,
\end{equation}
and  such that
\begin{equation}\label{stimw1}
\|w\|_{L^{\infty}(\Omega)}\le 
   C_1 \|g\|_{W^{1-\frac{1}{p}}_p(\partial\Omega)}
\end{equation}
where the constant $C_1$ depends on $\lambda_0$, and $\Omega\backslash D$.

Moreover, by defining
\begin{equation}\label{deff}
    f:=\mathrm{div}\big (\gamma \nabla w\big )\,,
\end{equation}
we also have
\begin{equation}\label{stimw2}
\|f\|_{L^{\infty}(\Omega)}\le 
   C_2 \|g\|_{W^{1-\frac{1}{p}}_p(\partial\Omega)}
\end{equation}
where the constant $C_2$  depends on $\lambda_0$, $\Omega\backslash D$ and on 
$\mathrm{dist}$$\bigl\{\partial\Omega, \partial D\bigr\}$.
\end{lemma}

\begin{proof}
By known results (see e.g. Grisvard th. $2.4.2.6$) the unique solution $w_0$ of the mixed problem  

\begin{equation*}
    \left\{
    \begin{aligned}
      \dive \left(\gamma_0 \nabla w_0\right)=0  &  \mbox{ in }\Omega\backslash D\\
      \gamma_0 \nabla w_0\cdot \nu=g    & \mbox{ on }\partial\Omega\\
      w_0=0    & \mbox{ on }\partial D
    \end{aligned}\right.
\end{equation*}

satisfies $w_0\in W^{2}_p(\Omega\backslash D)$ and its norm is bounded by the norm of $g$ in $W^{1-\frac{1}{p}}_p$. On the other hand, by Sobolev imbedding,

\[
W^{2}_p(\Omega\backslash D)\subset C^{1,\alpha}
(\overline{\Omega\backslash D})
\]

with continuous immersion and $\alpha=1-\frac{3}{p}>0$. We now take a smooth function $0\le \varphi\le 1$ with support in $\mathbb{R}^3\backslash {D}$ and such that
$\varphi=1$ in a neighborhood of $\partial\Omega$
and define

\begin{equation*}
w:=\varphi\,w_0
\end{equation*}

Clearly, $w\in C^{1,\alpha}
(\overline{\Omega})$, 
supp $w\subset \Omega\backslash D$, and the boundary condition \eqref{bdw} holds. Furthermore,

\begin{equation*}
\dive \left(\gamma \nabla w\right)=\dive \left(\gamma_0 \nabla w\right)=
\gamma_0\nabla w_0\cdot\nabla\varphi+
\nabla w_0\cdot \gamma_0\nabla\varphi+
w_0\dive \left(\gamma_0 \nabla \varphi\right).
\end{equation*}

Now, it is readily checked that 
$\dive \left(\gamma \nabla w\right)\in C^{0,\alpha}(\overline{\Omega})$
and that the estimates \eqref{stimw1}, \eqref{stimw2} are valid.
\end{proof}
We can now prove the $L^{\infty}$ bound of the weak solution to the problem \eqref{PN}:
\begin{proposition}\label{bound}
Let the assumptions of Proposition \ref{wellposedness} hold. Suppose further that $g\in W^{1-\frac{1}{p}}_p(\partial\Omega)$ with $p>3$ and that $\partial\Omega$, $\partial D$ are of class $C^{1,1}$.
Then the solution $u$ of the problem \eqref{PN} satisfies the following bound:
\begin{equation}\label{linfty}
    \|u\|_{L^{\infty}(\Omega)}\leq 2C_1\|g\|_{W^{1-\frac{1}{p}}_p(\partial\Omega)}+\left(\frac{C_2}{\alpha_0}\right)^{1/3}\|g\|^{1/3}_{W^{1-\frac{1}{p}}_p(\partial\Omega)}
\end{equation}

where the constants $C_1$, $C_2$ are the same as in Lemma \ref{probw}.
\end{proposition}

\begin{proof}  
Let us decompose the unique solution $u\in H^1(\Omega)$ to problem \eqref{PN} as follows:
\begin{equation*}
    u=v+w,
\end{equation*}
where $w$ is provided by Lemma \ref{probw}.

Then, by recalling the definition \eqref{deff}, $v$ (weakly) solves the problem
\begin{equation*}
    \left\{
    \begin{aligned}
      -\dive \left(\gamma \nabla v\right)+\alpha (v+w)^3=f  &  \mbox{ in }\Omega\\
      \gamma \nabla v\cdot \nu=0    & \mbox{ on }\partial\Omega
    \end{aligned}\right.
\end{equation*}
with $f\in C^{0,\alpha}(\overline{\Omega})$ defined in \eqref{deff}.
We provide a maximum principle for $v$ by an adaptation to the above nonlinear problem of Stampacchia's truncation method (see e.g. Brezis). For any $\phi\in H^1(\Omega)$ we have
\begin{equation}\label{PNvw}
   \int_{\Omega}\gamma\nabla v\nabla \phi+ 
   \int_{\Omega}\alpha (v+w)^3\phi=\int_{\Omega}f\phi
\end{equation}
Fix $G\in C^1(\mathbb{R})$ such that:
$G'(s)\le M$, $G$ is strictly increasing on $(0,+\infty)$,
$G(s)= 0$ $\forall s\le 0$.

Take now $\phi=G(v-|w_-|_{\infty}-\beta)$ in \eqref{PNvw},
where $|w_{\pm}|_{\infty}$ denote the $L^{\infty}$ norm of the positive and negative parts of $w$ and $\beta\in\mathbb{R}$. Hence, we get

\begin{align*}
    \int_{\Omega}G'(v-|w_-|_{\infty}-\beta)\gamma\nabla v\cdot \nabla v+
    & \int_{\Omega}\alpha (v+w)^3 G(v-|w_-|_{\infty}-\beta)\nonumber\\
    & =\int_{\Omega}f \,G(v-|w_-|_{\infty}-\beta)
\end{align*}

 Subtracting the term

 \[
\beta^3 \int_{\Omega}\alpha G(v-|w_-|_{\infty}-\beta)
 \]
  
to both members we obtain

\begin{align}\label{stamp1}
    \int_{\Omega}G'(v-|w_-|_{\infty}-\beta)\gamma\nabla v\cdot \nabla v+
& \int_{\Omega}\alpha\big [ (v+w)^3 -\beta^3\big ]G(v-|w_-|_{\infty}-\beta)\nonumber\\
    & =\int_{\Omega}\big (f-\beta^3 \alpha\big )  G(v-|w_-|_{\infty}-\beta)
\end{align}

Now by choosing

\begin{equation*}
   \beta=\left (\frac{|f_+|_{\infty}}{\alpha_0} \right)^{1/3}\,,
\end{equation*}
the right hand side of \eqref{stamp1} is $\le 0$.
Hence, the same must be true for the second term at the left hand side, since 
\begin{equation*}
   \gamma \nabla v\cdot \nabla v\ge \lambda_0 |\nabla v|^2
\end{equation*}
and $G'\ge 0$ in the first term. Then we get
\begin{equation*}
    \int_{\Omega}\alpha(v+w-\beta)d\, G(v-|w_-|_{\infty}-\beta) \le 0\,,
\end{equation*}
where
\[
d=d(v,w,\beta):=(v+w)^2+\beta(v+w)+\beta^2
\]
is strictly positive in $\Omega$. We can write the last inequality in the form
\begin{equation*}
    \int_{\Omega}\alpha(v-|w_-|_{\infty}-\beta)d\,
    G(v-|w_-|_{\infty}-\beta)+
    \int_{\Omega}\alpha(w+|w_-|_{\infty})
    d\, G(v-|w_-|_{\infty}-\beta)
    \le 0
\end{equation*}
Since the second integral is non negative, we finally get
\begin{equation*}
\int_{\Omega}\alpha\, d\,(v-|w_-|_{\infty}-\beta)
     G(v-|w_-|_{\infty}-\beta)
    \le 0
\end{equation*}
By the above inequality and by $t\,G(t)\ge 0$, with the strict inequality for $t>0$, we conclude
\begin{equation*}
v \le |w_-|_{\infty}+\beta=
|w_-|_{\infty}+\left(\frac{|f_+|_{\infty}}{\alpha_0} \right)^{1/3}
\end{equation*}

a.e. in $\Omega$. Then, the following bound holds for the (weak) solution $u$ of \eqref{PN}
\begin{equation*}
u=v+w \le |w_+|_{\infty}+
|w_-|_{\infty}+\left(\frac{|f_+|_{\infty}}{\alpha_0} \right)^{1/3}
\end{equation*}

In order to get an analogous lower bound, we change sign to both sides of \eqref{PNvw} and repeat the previous calculations by replacing 
$w\to -w$, $f\to -f$. Hence we obtain
\begin{equation*}
u=v+w \ge -|w_-|_{\infty}-
|w_*|_{\infty}-\left(\frac{|f_-|_{\infty}}{\alpha_0} \right)^{1/3}
\end{equation*}

From the two estimates above we get

\begin{equation*}
|u|_{\infty}\le 2|w|_{\infty}
+\left(\frac{|f|_{\infty}}{\alpha_0} \right)^{1/3}
\end{equation*}

Now the Proposition follows readily from the bounds \eqref{stimw1}, \eqref{stimw2}.
\end{proof}
\begin{remark}\label{gcont}
    By Sobolev embedding, the assumption $g\in W^{1-\frac{1}{p}}_p(\partial\Omega)$ with $p>3$ implies that $g\in C^{0,\delta}(\partial\Omega)$, where $\delta = 1-\frac{3}{p}>0$.
\end{remark}

\begin{remark}\label{locHol}
We stress that the previous estimates hold for any weak solution of problem \eqref{PN} with Neumann datum in $W_p^{1-\frac{1}{p}}$. Besides, it follows by the same bounds that the solution is locally H\"older continuous in $\Omega$.
In fact, we can now consider $u$ as a (weak) solution of the inhomogeneous elliptic equation
\[
\dive \left(\gamma \nabla u\right)=\alpha u^3
\]
where the right hand side is essentially bounded in $\Omega$ by Proposition \ref{bound}. Hence, by the known regularity results of De Giorgi-Nash, we can bound the local
$C^{0,\delta}$ norm of $u$ (for suitable $\delta>0$) with
$\|g\|_{W^{1-\frac{1}{p}}_p}(\partial\Omega)$
\end{remark}

We now show that the solution of $\eqref{PN}$ is actually a Holder continuous function in $\overline{\Omega}$.
\begin{proposition}\label{regularity}
With the assumptions of Proposition \ref{bound},
    there exist $0<\delta<1$ and $C>0$, depending on $\lambda_0$, $\alpha_0$, $\Omega$, $D$ and on $\|g\|_{W^{1-\frac{1}{p}}_p(\partial\Omega)}$ such that if $u$ is the solution of \eqref{PN}, then
    \begin{equation}\label{regest}
        \|u\|_{C^{0,\delta}\left(\overline{\Omega}\right)}\leq C \end{equation}
\end{proposition}
\begin{proof}
    Let $D\subset\Omega'\subset\subset\Omega$. By remark \ref{locHol}, the estimate \eqref{regest} holds on $\overline{\Omega}'$ with a constant $C'$ also depending on $\Omega'$. Now, pick a smooth function $\chi$ with support in $\mathbb{R}^3\backslash {D}$ and such that $\chi=1$ in a neighborhood of $\Omega\backslash\Omega'$ and define

\begin{equation*}
\tilde u:=\chi\,u
\end{equation*}

Clearly, $\tilde u\in H^1(\Omega)$ and vanishes on $D$. Then, by calculations as in the proof of Lemma \ref{probw}, 

\begin{equation}\label{equtilde}
\dive(\gamma\nabla\tilde u)=
\chi\alpha u^3+
u\dive(\gamma_0\nabla\chi)+\gamma_0\nabla\chi\cdot\nabla u+
\gamma_0\nabla u\cdot\nabla\chi
\end{equation}

Note that the first two terms at the right hand side are bounded (by Proposition \ref{bound}) while the other terms belong to $L^2(\Omega)$. Hence, $\dive(\gamma\nabla\tilde u)\in L^2(\Omega)$; moreover, $\tilde u$ has the same Neumann datum $g$ as $u$. Again by Sobolev imbedding, $g\in H^{1/2}(\partial\Omega)$ (actually, $g\in H^{s}(\partial\Omega)$ with $s=\frac{1}{2}+\frac{3(p-2)}{2p}$), so that we conclude
$\tilde u\in H^2(\Omega)$ and the same is true for $u$ in $\Omega\backslash\Omega'$. Now, since $H^2(\Omega\backslash\Omega')\subset C^{0,1/2}(\Omega\backslash\Omega')$, we easily get \eqref{regest} in $\Omega\backslash\Omega'$, with another constant $C''$ (depending on $\Omega\backslash\Omega'$ and on $D$) and where $\delta=1/2$. It remains to show that \eqref{regest} holds in the whole $\Omega$ with a constant $C$ as in the statement of the theorem.  We first note that such an estimate holds on $D$ by the previous arguments relying on the regularity results for the elliptic equations in divergence form. Let us now consider
an open cover of $\Omega\backslash D$; by suitably choosing $\Omega'$, we can consider an open set $\mathcal{U}$ of the cover
as a subset of $\Omega'$ or of $\Omega\backslash\Omega'$. In the former case, we have the bound\eqref{regest} on $\mathcal{U}$ with a constant $C'$ from the remark \ref{locHol}; in the latter case, the bound holds with a constant $C''$ calculated as above in terms of the $H^2$ norm of $u$ in a neighborhood of $\partial\Omega$ that includes $\mathcal{U}$. We further note that $C'$ is uniformly bounded if the
 subsets $\Omega'$ remain at a finite distance from $\partial\Omega$, while $C''$ is uniformly bounded if $\Omega\backslash \Omega'$ remain at a finite distance from $\partial D$. It follows that 
$\sup\big (\min\{C', C''\}\big )<\infty$ where the supremum is taken on any family of neighborhoods covering $\Omega\backslash D$
(we set $C''=\infty$ for the neighborhoods intersecting $\partial D$ and $C'=\infty$ for those intersecting $\partial\Omega$). Taking $C$ as the maximum between this supremum and the constant of the bound on $D$, the proposition is proved.
\end{proof}

\begin{remark}\label{uw2p}
    By iterating the arguments following \eqref{equtilde}, one can show that $u\in W^2_p(\Omega\backslash\Omega')$ for any $\Omega'$ such that $D\subset \Omega'\subset\subset \Omega$.
\end{remark}

\section{Linearization of the Neumann to Dirichlet map}

In this section, we prepare the nonlinear inverse problem of determining the anisotropic conductivity by passing from the nonlinear NtoD map
$\mathcal N^{\gamma,\alpha}_{NL}$ to its first-order linearization.
The key idea is to perturb a fixed non negative and non trivial  smooth enough Neumann datum $g_{0}$, 
analyze the behaviour of the corresponding solutions $u(\cdot,g_{0}+\tau g^{*})$ as $\tau \to 0$, 
and identify the Fr\'echet derivative of $\mathcal N^{\gamma,\alpha}_{NL}$.
To do this, we first establish some preliminary results providing uniform lower bounds for the solutions of the nonlinear problem corresponding to nonnegative Neumann data.
We then show that this derivative coincides with the NtoD map associated with the 
linear Schr\"odinger-type operator
\begin{equation}\label{L}
L  = -\operatorname{div}(\gamma \nabla \cdot) + 3\alpha u_0^{2} ,
\end{equation}
where $u_0:=u(\cdot,g_0)$ is the solution to \eqref{PN} corresponding to a nontrivial background current $g_{0}\geq $.
This reduction is essential: equality of the nonlinear maps for two conductivities 
implies equality of their linearised maps, which in turn allows us to apply 
boundary determination and unique continuation arguments in the subsequent sections.


Let us first show that, if 
\begin{equation}\label{hpg0}
g_0\in W^{1-\frac{1}{p}}_p(\partial\Omega)\subset C^0(\partial\Omega),\quad g_0\ge 0 \quad\mbox{ and }\quad g_0\not\equiv 0 \mbox{ on }\partial\Omega,
\end{equation}
then $u_0$ is strictly positive in $\overline{\Omega}$.

\begin{proposition}\label{defc0}
    With the assumptions of Proposition \ref{bound}, let $g_0$ be as in \eqref{hpg0}. Then, there exists $c_0>0$ such that
    $u_0=u(\cdot,g_0)\geq c_0\mbox{ in }\Omega$.
\end{proposition}

\begin{proof}
Let us write the equation for $u_0$ as
\begin{equation*}
    \dive( \gamma\nabla u_0) - c(x)\, u_0=0
\end{equation*}
where 
\begin{equation*}
    c(x):=\alpha u_0^2\geq 0\mbox{ in }\Omega.
\end{equation*}
Notice that, by Proposition \ref{regularity}, $u_0$ is continuous in $\overline{\Omega}$ and therefore $c$ is non negative and bounded.

Then, $u_0$ is a (weak) solution of a linear elliptic equation satisfying the strong maximum and minimum principle in $H^1$
(see Theorem 8.19 in \cite{GT} and the subsequent remarks). It follows in particular that
$u_0$ cannot have a non-positive minimum inside $\Omega$, because it cannot be constant since $g_0\not\equiv 0$. Moreover, since $u_0\in W^2_p$ in any neighborhood of $\partial\Omega$ contained in $\Omega\backslash D$  (see Remark \ref{uw2p})
and solves there the elliptic equation
\begin{equation*}
  \dive( \gamma_0\nabla u_0) - c(x)\, u_0=0
\end{equation*}
with $\nabla u_0\cdot\nu=g_0\ge 0$ on $\partial\Omega$, we can apply Hopf Lemma (Lemma 3.4 in \cite{GT}) to conclude that $u_0$ cannot have a non-positive minimum on $\partial\Omega$ as well. This means that $c_0=\min_{\Omega}u_0$ is strictly positive.\end{proof}

 With the same assumptions, one can also achieve some quantitative estimate on $c_0$ by exploiting the weak Harnack inequality and the regularity of a solution in a neighborhood of the boundary $\partial\Omega$.

\begin{remark}\label{solproblin}
From the proof of Propositions \ref{defc0}, it is clear that the same estimates from below hold for a weak solution $v$ of the linear elliptic problem
  \begin{equation}\label{LN}
    \left\{
    \begin{aligned}
      -\dive \left(\gamma \nabla v\right)+c(x)\,v=0  &  \mbox{ in }\Omega\\
      \gamma \nabla v\cdot \nu=g^*    & \mbox{ on }\partial\Omega
    \end{aligned}\right.
\end{equation}
where $c(x)\ge 0$ is bounded ($c\neq 0$), $g^*$ is smooth enough and $g^*>0$. Note that condition $c\neq 0$ ensures the unique solvability of the above problem in $H^1(\Omega)$ by the Fredholm alternative
(see, e.g. \cite{GT}, eq. (8.77) and the subsequent discussion). Moreover, the $L^{\infty}$ estimates and the regularity properties of the solution  obtained in the previous section can also be easily proved for the solution $v$
of \eqref{LN}.
\end{remark}

We now state and prove some preliminary lemmas in the spirit of $\cite{isakov1994global}$

\begin{lemma}\label{lemma1}
    Let $g^*\in H^{-1/2}(\partial\Omega)$, let $g_0$ satisfy \eqref{hpg0}, and $\tau\in\mathbb{R}$, then
    \begin{equation}\label{thlemma1}
    \lim_{\tau\to 0}\|u(\cdot,g_0+\tau g^*)-u(\cdot,g_0)\|_{H^1(\Omega)}=0
    \end{equation}
\end{lemma}
\begin{proof}
    Denote by $u_\tau=u(\cdot,g_0+\tau g^*)$ and $u_0=u(\cdot,g_0)$ and let us define
\begin{equation*}
    z=u_\tau-u_0.
\end{equation*}
Notice that 
\begin{equation}\label{bdz}
    \gamma \nabla z\cdot \nu=\tau g^* \mbox{ on }\partial\Omega
\end{equation}
and
\begin{equation}\label{eqz}
    \dive(\gamma\nabla z)=\alpha(u_\tau^3-u_0^3)= \alpha q_0 z, 
\end{equation}
where $q_0=(u_\tau^2+u_0^2+u_\tau u_0)$.
By multiplying \eqref{eqz} by $z$ and integrating by parts we have, by \eqref{bdz},
\begin{equation}\label{zparts}
    \int_\Omega \gamma\nabla z\cdot \nabla z+\int_\Omega \alpha q_0 z^2=\tau\int_{\partial\Omega}z g^*\leq |\tau|\|z\|_{H^1(\Omega)}\|g^*\|_{H^{-1/2}(\partial\Omega)}
\end{equation}
By Proposition \ref{defc0}, there exists a positive number $c_0$ 
\begin{align}
 \label{qzero}
    q_0&=u_\tau^2+u_0^2+u_\tau u_0= (u_\tau+\frac{1}{2}u_0)^2+\frac{3}{4}u_0^2\geq \frac{3}{4}c_0^2
\end{align}
By \eqref{ell}, \eqref{hpalpha} and \eqref{qzero}, the inequality \eqref{zparts} gives

\begin{equation*}
    \lambda_0\|\nabla z\|^2_{L^2(\Omega)}+\frac{3\alpha_0c_0^2}{4}\| z\|^2_{L^2(\Omega)}\leq |\tau|\|z\|_{H^1(\Omega)}\|g^*\|_{H^{-1/2}(\partial\Omega)}.
\end{equation*}
and, hence,
\begin{equation*}
    \|z\|_{H^1(\Omega)}\leq \frac{1}{\min\left\{\lambda_0, \frac{3\alpha_0c_0^2}{4}\right\}}|  \tau| \|g^*\|_{H^{-1/2}(\partial\Omega)}
\end{equation*}
which gives \eqref{thlemma1}.

\end{proof}
Let us now denote by $v_0$ the unique solution of the Neumann boundary value problem
\begin{equation}\label{PNL}
    \left\{\begin{aligned}
      -\dive \left(\gamma \nabla v_0\right)+3\alpha u_0^2v_0=0  &  \mbox{ in }\Omega\\
      \gamma \nabla v_0\cdot \nu=g^*    & \mbox{ on }\partial\Omega.
    \end{aligned}\right.
\end{equation}

This solution exists by Remark \ref{solproblin} (because $3\alpha u_0^2$ is bounded and larger than $3\alpha_0c_0^2>0$) and satisfies the estimate
\begin{equation*}
    \|v_0\|_{H^1(\Omega)}\leq C_0 \|g^*\|_{H^{-1/2}(\partial\Omega)}
\end{equation*} 
where $C_0$ also depends on $c_0$ and, hence, on $u_0$. By the same remark we infer that if $g^*\in W_p^{1-\frac{1}{p}}$ then $v_0$ is bounded and continuous in $\overline{\Omega}$.
\begin{lemma}\label{lemma2}
Let $g_0$, $g^*$, $u_0$, $u_\tau$ be given as in Lemma \ref{lemma1} and assume further that $g^*\in W_p^{1-\frac{1}{p}}$. 
Let $v_0$ be the solution of \eqref{PNL} with Neumann datum $g^*$. Then,
\begin{equation*}
    \lim_{\tau\to 0}\left\|\frac{u_\tau-u_0}{\tau}-v_0\right\|_{H^1(\Omega)}=0    
\end{equation*}
\end{lemma}
\begin{proof}
    Let us define
 \begin{equation*}
        w=\frac{u_\tau-u_0}{\tau}-v_0
    \end{equation*}
    and notice that 
\begin{equation*}
    \gamma \nabla w\cdot \nu= 0 \mbox{ on }\partial\Omega.
\end{equation*}

Moreover
\begin{equation}\label{eqv}
    \dive(\gamma\nabla w)-\alpha q_0 w =\alpha v_0 (u_\tau+2u_0)(u_\tau-u_0), 
\end{equation}    
where (see the proof of Lemma \ref{lemma1})
\begin{equation*}
    q_0=u_\tau^2+u_0u_\tau+u_0^2\geq \frac{3}{4}c_0^2
\end{equation*}
Furthermore, by our assumptions and by Proposition \ref{bound}, we have
\[
\|\alpha v_0 (u_\tau+2u_0)\|_{L^\infty(\Omega)}\le C
\]
with $C$ independent of $\tau$, so by multiplying the equation \eqref{eqv} by $w$ and after integration by parts, we get
\begin{align*}
    \int_\Omega \gamma\nabla w\cdot \nabla w+\int_\Omega \alpha q_0 w^2=&\int_{\Omega}\alpha v_0(u_\tau+2u_0)(u_0-u_\tau)w \nonumber\\
    &\leq \|\alpha v_0 (u_\tau+2u_0)\|_{L^\infty(\Omega)}\|u_0-u_\tau\|_{L^2(\Omega)}\|w\|_{L^2(\Omega)}
    \nonumber \\
    &\le C_\epsilon \|u_0-u_\tau\|^2_{L^2(\Omega)}
    +\epsilon \|w\|^2_{L^2(\Omega)}
\end{align*}
for any $\epsilon>0$. Hence, arguing as in the proof of the preceding lemma, we get the bound
\begin{equation*}
    \|w\|_{H^1(\Omega)}\le C\,\|u_0-u_\tau\|_{L^2(\Omega)}
\end{equation*}
so that the lemma follows by equation \eqref{thlemma1}.
\end{proof}
The result of Lemma \ref{lemma2} implies the following
\begin{theorem}\label{teoder}
Let $\Omega$ be a bounded domain, $D\subset\subset \Omega$ such that $\partial\Omega$, $\partial D$ are of class $C^{1,1}$. Let
$\gamma$ be given by \eqref{gamma} and satisfying \eqref{ell}, and $\alpha\in L^{\infty}(\Omega)$ satisfying  \eqref{hpalpha}.
Let $g^*\in W_p^{1-1/p}(\partial\Omega)$ and $g_0$ satisfy \eqref{hpg0}. 

Then
\begin{equation*}
   \lim_{\tau\to 0}\frac{\mathcal{N}^{\gamma,\alpha}_{NL} (g_0+\tau g^*)- \mathcal{N}^{\gamma,\alpha}_{NL} (g_0)}{\tau}=\mathcal{N}^{\gamma,3\alpha u_0^2}_{L} (g^*) 
\end{equation*}
where $\mathcal{N}^{\gamma,3\alpha u_0^2}_{L}$ is the Neumann to Dirichlet map associated to the operator $L$ defined in \eqref{L}.
\end{theorem}
In particular, Theorem \ref{teoder}, the density of $W^{1-1/p}_p(\partial\Omega)$ in $H^{-1/2}(\partial\Omega)$ (with respect to the $H^{-1/2}$-norm)
 and continuity of the linear NtoD map $\mathcal{N}^{\gamma,3\alpha u_0^2}_{L}$ imply the following
\begin{corollary}\label{corder}
The knowledge of $\mathcal{N}^{\gamma,\alpha}_{NL} (g)$ for every $g\in H^{-1/2}(\partial\Omega)$ implies the knowledge of 
$\mathcal{N}^{\gamma,3\alpha u_{0}^2}_{L} (g^*)
$ for every $g^*\in H^{-1/2}(\partial\Omega)$ and with $u_0=u(\cdot,g_0)$ for $g_0$  satisfying \eqref{hpg0}.\end{corollary}

\section{Uniqueness of the inverse problem for the anisotropic case}
In the discussion of the inverse problem, we assume that the nonlinear ionic coefficient
$\alpha$ is known in $\Omega$, and we focus on the unique determination of the anisotropic
conductivity tensor $\gamma$ of the form~\eqref{gamma} from the NtoD map
\(N^{NL}_{\gamma,\alpha}\). 

In this section we need an additional geometric assumption. Namely, we will require that both $\partial \Omega$ and $\partial D$ contain a nonflat portion.
The rigorous definition is given below.
This assumption is not restrictive since both $\Omega$ and $D$ are bounded domains with boundaries of class $C^{1,1}$.
\begin{definition}\label{nonflat}
We say that a portion $\Sigma$ of  a $C^{1,1}$ surface is non-flat, if  there exists a point $P \in \Sigma$ such that, considering the reference system and the function $\varphi$ as in Definition \ref{regbordo}, we have that $\varphi$ is not identically zero near $P=0$.
\end{definition}

The main result of this section is the following:
\begin{theorem}\label{unicis}
Let $D\subset\subset\Omega$ be given bounded domains in $\mathbb{R}^3$ with $C^{1,1}$ boundaries. Let us assume that both $\partial \Omega$ and $\partial D$ have non flat portions. Let $\gamma$ and $\alpha$ satisfy assumptions \eqref{gamma}, \eqref{ell},  and \eqref{hpalpha}.
Then, the knowledge of $\mathcal{N}^{\gamma,\alpha}_{NL} (g)$ for every $g\in H^{-1/2}(\partial\Omega)$ uniquely determines $\gamma$.
\end{theorem} 
\medskip
The proof of the Theorem \ref{unicis} relies on the combination of several partial results. 

\subsection{Notation and useful definitions}
Let us first introduce some useful definitions and notation.
Consider the metric
\begin{equation*}
    \mathfrak{g}=(\mbox{det}\ {\gamma}){\gamma}^{-1}
\end{equation*}

Let $N_{\gamma_0}$ be the Neumann kernel for the operator $\mbox{div}(\gamma_0\nabla \cdot)$, with constant $\gamma_0\in \emph{Sym}_3$ on the half space $\mathbb{R}^3_+$. For $y\in \mathbb{R}^3_+$ $N_{\gamma_0}$ is the distributional solution to 
\begin{equation}\label{NeumannKernelsemispazioint}
\begin{cases}
\dive({\gamma_0}\nabla N_{\gamma_0}(\cdot,y))=-\delta(\cdot-y) & \mbox{ in }\mathbb{R}^3_+ ,\\[3pt]
\gamma_0 \nabla N_{\gamma_0}(\cdot,y)\cdot \nu=0 & \mbox{ on }\Pi_3,\\[3pt]
N_{\gamma_0}(x,y)\rightarrow 0  &  \mbox{ as}\ \ |x|\rightarrow +\infty,
\end{cases}
\end{equation} 
 where $\Pi_3=\{x=(x',x_3)\in\mathbb{R}^3: \ x_3=0\}$. 
For $y'\in \Pi_3$, we consider the boundary extension of the solution of \eqref{NeumannKernelsemispazioint} , $N_{\gamma_0}(\cdot,y')$, which turns out to be the distributional solution to
\begin{equation}\label{NeumannKernelsemispaziobd}
\begin{cases}
\dive({\gamma_0}\nabla N_{\gamma_0}(\cdot,y'))=0 & \mbox{ in }\mathbb{R}^3_+ ,\\[3pt]
\gamma_0 \nabla N_{\gamma_0}(\cdot,y')\cdot \nu=\delta(\cdot-y') & \mbox{ on }\Pi_3,\\[3pt]
N_{\gamma_0}(x,y')\rightarrow 0  &  \mbox{ as}\ \ |x|\rightarrow +\infty,
\end{cases}
\end{equation} 
and for every $x\in \mathbb{R}^3_+$ and $y'\in \Pi_3$ we have 
\begin{equation}\label{Nsemip}
N_{\gamma_0}(x,y')=\frac{1}{2\pi}\:\Big({\mathfrak{g}_0}(x-y')\cdot(x-y')\Big)^{-\frac{1}{2}},\end{equation}
where 
\[\mathfrak{g}_0=(\det\gamma_0)\gamma_0^{-1}.\]
In particular, if $N_{\gamma_0}(x',y')$ is known for every $x',y'\in \Pi_3$, then $\mathfrak{g}_0\xi\cdot \xi$ is known fo every $\xi\in \Pi_3$.

Recall that, for $y\in\Omega$, the Neumann kernel $N^{\Omega}_{\gamma, q}(\cdot,\:y)$ for the Schr\"{o}dinger type operator $\mbox{div}(\gamma\nabla ) -q$ in $\Omega$ is defined to be the distributional solution to
\begin{equation*}
\begin{cases}
\mbox{div}(\gamma(\cdot)\nabla N^{\Omega}_{\gamma, q}(\cdot,y)) - q(\cdot) N^{\Omega}_{\gamma, q}(\cdot,y) =-\delta(\cdot -y)  &\mbox{ in }\Omega,\\[3pt]
\gamma (\cdot)\nabla N^{\Omega}_{\gamma, q}(\cdot, y)\cdot \nu=0  &\mbox{ on }\partial\Omega.\end{cases}
\end{equation*}
$N^{\Omega}_{\gamma, q}$ extends continuously up to the boundary $\partial\Omega$ and  for $y'\in\partial\Omega$, it solves

\begin{equation*}
\begin{cases}
\mbox{div}(\gamma(\cdot) \nabla N^{\Omega}_{\gamma, q}(\cdot,y')) - q(\cdot)N^{\Omega}_{\gamma, q}(\cdot,y') =0 & \mbox{ in }\Omega\\[3pt]
\gamma (\cdot)\nabla N^{\Omega}_{\gamma, q}(\cdot, y')\cdot \nu=\delta(\cdot -y')
& \mbox{ on }\partial\Omega,
\end{cases} 
\end{equation*}
where here $\delta$ is the $2-$ dimensional Dirac delta.

The existence of $N^{\Omega}_{\gamma, q}(\cdot,\:y)$ is a consequence of the results in \cite{kim2024neumann}, where the authors prove existence of the Neumann kernel for elliptic operators of the form
\[
 -\operatorname{div}(A\nabla u) + d u
\]
in bounded Lipschitz domains, with $A$ uniformly elliptic and bounded and with
a strictly positive potential $d \ge d_0 > 0$. 

\subsection{Main steps of the proof of Theorem \ref{unicis}.}
The proof combines the linearization result of Section 3 with the boundary determination
for anisotropic conductivities and an interior unique continuation argument. We assume
throughout that $\gamma$ and $\alpha$ satisfy \eqref{gamma}–\eqref{hpalpha}, that
$\partial\Omega$ and $\partial D$ are of class $C^{1,1}$, and that $\partial\Omega$ and $\partial D$ contain  non-flat open portions in the sense of Definition~\ref{nonflat}. \\

\textbf{Premise.}
By Proposition~2.1 the nonlinear Neumann problem \eqref{PN} is well posed for every
$g\in H^{-1/2}(\partial\Omega)$.
Propositions~2.3 and~2.6 provide $L^\infty$ and $C^{0,\delta}$ bounds for the solution
$u(\cdot,g)$ if $g\in W^{1-1/p}_p(\partial\Omega)$.
Moreover, by Proposition~3.1 (and Proposition~3.2), for a fixed nonnegative and non trivial datum
$g_0\in W^{1-1/p}_p(\partial\Omega)$ the corresponding solution
$u_0=u(\cdot,g_0)$ is strictly positive in $\Omega$ with a quantitative lower bound.

As a consequence, the linearized operator
\[
L = -\operatorname{div}(\gamma\nabla\cdot) + 3\alpha u_0^2
\]
has a bounded strictly positive zeroth–order coefficient.
Furthermore, Theorem~3.7 shows that the Fréchet derivative of the nonlinear
Neumann-to-Dirichlet map at $g_0$ coincides with the linear NtD map
$\mathcal{N}_{L}^{\gamma,q}$ where $q=3\alpha u_0^2$ and Corollary~3.8 shows that  $\mathcal{N}_{NL}^{\gamma ,\alpha}$ determines $\mathcal{N}_{L}^{\gamma,q}$.

\medskip

We now describe the main steps of the proof.

\begin{enumerate}

\item[Step 1.]
\emph{Determination of $\gamma$ on $\partial\Omega$ and in $\Omega\setminus D$.}
From the linear NtD map we first recover the boundary values of the Neumann kernel
$N^{\Omega}_{\gamma ,q }$ via Lemma~4.2, which expresses the bilinear form associated
with $\mathcal{N}_{L}^{\gamma, q }$ as a double integral over
$\partial\Omega\times\partial\Omega$.

Lemma~4.3 provides an asymptotic expansion of $N^{\Omega}_{\gamma,q}(x,y)$ near a boundary
point $y'\in\partial\Omega$,
\[
N^{\Omega}_{\gamma,q}(x,y')
=
\frac{1}{2\pi}
\Big(\mathfrak{g}_0 (x-y')\cdot(x-y')\Big)^{-1/2}
+
\mathcal{O}(|\ln|x-y'||),
\quad
x\to y',\ x\in\overline{\Omega}\setminus\{y'\}.
\]
Using this expansion, Lemma~4.4 shows that the leading term determines the tangential
component of the metric
\[
\tilde{\mathfrak{g}}_{ij}
=
\mathfrak{g}_0 e_i\cdot e_j,
\qquad i,j=1,2,
\]
where $\{e_1,e_2\}$ is an orthonormal basis of the tangent plane at $y'$ in the direction of $e_1$ and $e_2$
Lemma~4.5, together with the non-flatness assumption, allows us to recover the full
metric $\mathfrak{g}_0$ on $\partial\Omega$ and hence $\gamma=\gamma_0$ on $\partial\Omega$ .
Since $\gamma$ is constant in $\Omega\setminus D$, it follows that $\gamma=\gamma_0$ in $\Omega\setminus D$.

\item[Step 2.]
\emph{Identification of the potential $q$ in $\Omega\setminus D$.}
With $\gamma^{(1)}=\gamma^{(2)}=\gamma_0$ in $\Omega\setminus D$, let $u_0^{(i)}$ be the
background solutions corresponding to $g_0$ for $(\gamma^{(i)},\alpha)$, and set
$w=u^{(1)}_{0}-u^{(2)}_{0}$.
Equality of the nonlinear NtD maps implies that $u^{(1)}_{0}$ and $u^{(2)}_{0}$ have the same
Cauchy data on $\partial\Omega$.
Subtracting the equations satisfied by $u^{(1)}_{0}$ and $u^{(2)}_{0}$ in $\Omega\setminus D$
yields
\[
-\operatorname{div}(\gamma\nabla w) + \alpha p(x) w = 0
\quad \text{in } \Omega\setminus D,
\]
where
$p(x)=(u^{(1)}_{0})^2+u^{(1)}_{0}u^{(2)}_{0}+(u^{(2)}_{0})^2\in L^\infty(\Omega)$.
By unique continuation, $w\equiv0$ in $\Omega\setminus D$, hence
\[
q^{(1)}=3\alpha u^{(1)}_{0} =3\alpha u^{(2)}_{0}=q^{(2)}
\quad \text{in } \Omega\setminus D.
\]

\item[Step 3.]
\emph{From the linear NtD map $\mathcal{N}_{L}^{\gamma, q}$ on $\partial\Omega$ to the linear NtD map on $\partial D$, $\mathcal{N}_{L, \partial D}^{\gamma, q}$.}

From 
\begin{equation*}
    {\mathcal{N}}_{L}^{\gamma^{(1)}, q^{(1)}}={\mathcal{N}}_{L}^{\gamma^{(2)}, q^{(2)}}
\end{equation*}
 the fact that $\gamma^{(1)}=\gamma^{(2)}=\gamma_0$ and $q^{(1)}=q^{(2)}$ in $\Omega\setminus D$
and unique continuation properties of solutions to linear elliptic equations
it follows that
\begin{equation*}
    {\mathcal{N}}_{L, \partial D}^{\gamma^{(1)}, q^{(1)}}={\mathcal{N}}_{L, \partial D}^{\gamma^{(2)}, q^{(2)}}
\end{equation*}


\item[Step 4.]
\emph{Recovery of $\gamma$ inside $D$ and conclusion.}
Repeating the boundary determination argument of Step~1 with $D$ in place of $\Omega$,
we reconstruct $\gamma$ on $\partial D$ and hence in $D$, since $\gamma$ is constant
there.
Together with Step~1 this yields $\gamma_1=\gamma_2$ in all of $\Omega$, completing the
proof of Theorem~\ref{unicis}.

\end{enumerate}

\subsection{Auxiliary results}
\begin{lemma}
If the linear Neumann-to-Dirichlet map $\mathcal{N}^{\gamma, q}_{L}$   is known, then the Neumann kernel
$N^\Omega_{\gamma,q}(x,y)$ is known for every $x,y\in\partial\Omega$, $x\neq y$

\end{lemma}

\begin{proof}
Let $\varphi\in C^{0,1}(\partial\Omega)\cap H^{-1/2}(\partial\Omega)$ and let $u\in H^1(\Omega)$ be the
(unique) weak solution of
\begin{equation*}
\begin{cases}
- \operatorname{div}(\gamma\nabla u) + q u = 0 & \text{in }\Omega,\\[2mm]
\gamma\nabla u\cdot\nu = \varphi & \text{on }\partial\Omega.
\end{cases}
\end{equation*}
By the representation formula for Neumann problems with kernel $N^\Omega_{\gamma,q}$, \cite{colton1998inverse}, we have for every $\xi\in\partial\Omega$
\[
u(\xi)
= \int_{\partial\Omega} N^\Omega_{\gamma,q}(\xi,\eta)\,\varphi(\eta)\,d\sigma(\eta).
\]
By definition of the linear Neumann-to-Dirichlet map, $u|_{\partial\Omega} = N^L_{\gamma,q}\varphi$,
hence
\begin{equation*}
\mathcal{N}^L_{\gamma,q}\varphi(\xi)
= \int_{\partial\Omega} N^\Omega_{\gamma,q}(\xi,\eta)\,\varphi(\eta)\,d\sigma(\eta),
\qquad \xi\in\partial\Omega.
\end{equation*}

Then for any $\psi\in C^{0,1} (\partial\Omega) \cap H^{-\frac{1}{2}}(\partial\Omega)$ it follows 
\begin{equation}\label{nucleo}
 \langle \psi, {\mathcal{N}}_{L}^{\gamma, q} \varphi\rangle_{\partial\Omega} = \int_{\partial\Omega}\psi(\xi) \mbox{d}\sigma(\xi) \int_{\partial\Omega} N^{\Omega}_{\gamma, q}(\xi,\eta) \varphi(\eta) \mbox{d}\sigma(\eta)=\int_{\partial\Omega \times \partial\Omega} N^{\Omega}_{\gamma, q}(\xi,\eta) \psi(\xi)\varphi(\eta) \mbox{d}\sigma(\xi)\mbox{d}\sigma(\eta)
\end{equation}
By choosing in \eqref{nucleo}
\begin{equation*}
\psi(\xi)= \delta_{\epsilon}(\xi,x) \quad
\varphi(\eta)=\delta_{\epsilon}(\eta, y)
\end{equation*}
where $\delta_{\epsilon}$ are approximate Dirac's delta function on $\partial\Omega$ centered on the second argument,  and by letting $\epsilon \rightarrow 0$  we determine $ N^{\Omega}_{\gamma, q}(x,y)$ for $x, y\in\partial\Omega$.
\end{proof}

In the following results, leading to Proposition \ref{lemma4.5}, we show that the knowledge of $N^{\Omega}_{\gamma,q}$ determines $\gamma$ in $\Omega\setminus D$. This part is true in the slightly more general assumptions that $\gamma$ is constant in a neighborhood of a non-flat portion of $\partial \Omega$.

\begin{lemma}\label{theorem neumann function holder}
Let $y'\in\partial\Omega$ and suppose there exists a neighborhood $\mathcal{U}$ of $y'$
such that $\partial\Omega\cap \mathcal{U}$ is $C^{1,1}$ and  $\gamma=\gamma_0$ in $\mathcal{U}\cap\Omega$.
Then, given $q\in L^{\infty}(\Omega)$ with $q>0$ in $\Omega$, we have 
\begin{equation*}
N^{\Omega}_{\gamma,q}(x,y')=
\frac{1}{2\pi}\:\Big({\mathfrak{g}_0}(x-y')\cdot(x-y')\Big)^{-\frac{1}{2}}+
\mathcal{O}(|\ln|x-y'||),
\end{equation*}
as $x\rightarrow y'$, $x\in\overline{\Omega}\setminus{\{y'\}}$. 
\end{lemma}

\begin{proof}
The proof is inspired by Theorem 3.4 in \cite{Al-dH-G} and by \cite{Miranda}.

We may assume with no loss of generality that $y'=0$. Let $\rho>0$ be such that   $\gamma=\gamma_0$ in  $B_{2\rho}(y')\cap \Omega$.

Let us consider a test function $\psi\in C^{0,1}_0(B_{\rho})$, then we have 
\begin{eqnarray}\label{1}
    \int_{\Omega\cap B_{\rho}} \gamma_0\nabla_x N^{\Omega}_{\gamma,q}(x,0)\cdot \nabla_x \psi (x) \mbox{d}x+ \int_{\Omega\cap B_{\rho}} q(x)N^{\Omega}_{\gamma,q}(x,0)\psi (x)\mbox{d}x 
    = \psi(0).
\end{eqnarray}
We introduce the change of variable 
\begin{equation}\label{ChVa}
    \begin{cases}
       z'=x'  \\
     z_n=x_n-\varphi(x')
    \end{cases}
\end{equation}
where $\varphi$ is the function introduced in Definition \ref{regbordo}.

By the regularity hypothesis on the boundary we have that 
\begin{equation*}
  z= x + \mathcal{O}(|x'|^{2})  \ .
\end{equation*}
We define now the Jacobian of the change of variables $z$ as $J:=\dfrac{\partial z}{\partial x}$ and deduce that 
\begin{equation*}
J= I + \mathcal{O}(|x'|) \ .
\end{equation*}
We define 
\begin{eqnarray*}
&&\tilde{\gamma}(z)=\left(\frac{1}{\mbox{det}(J)} J\gamma J^T \right)(x(z))\\
&&\tilde{q}(z)=q(x(z))\\
&&{\tilde{N}}_{\gamma, q}^{\Omega}(z)= {{N}}_{\gamma, q}^{\Omega}(x(z),0) \ .
\end{eqnarray*}
We obtain by \eqref{1} 
\begin{eqnarray*}
    \int_{\{z_n>0\}} \tilde{\gamma}(z)\nabla_z \tilde{N}_{\gamma,q}^{\Omega}(z)\cdot \nabla_z\psi(x(z))\mbox{d}z + \int_{\{z_n>0\}} \tilde{q}(z) \tilde{N}_{\gamma,q}^{\Omega}(z)\psi(x(z))\mbox{d}z=\psi(0),
\end{eqnarray*}
and have that 
\begin{eqnarray}\label{sigmalocale}
    \tilde{\gamma}(z)=\gamma_0 + \mathcal{O}(|z'|).
    \end{eqnarray}
Let $N_{\gamma_0}$ be the distributional solution introduced in \eqref{NeumannKernelsemispaziobd} and define 
\begin{equation*}
R(z)= \tilde{N}_{\gamma,q}^{\Omega}(z) - N_{\gamma_0}(z,0).
\end{equation*}
Define $B_{\rho}^+:=B_{\rho}\cap\mathbb{R}^3_+$. The function $R$ satisfies 
\begin{equation}\label{eqR}
\begin{cases}
\mbox{div}(\gamma_0\nabla R(z)) = \mbox{div} ((\gamma_0- \tilde{\gamma}(z))\nabla \tilde{N}_{\gamma,q}^{\Omega}(z))+ \tilde{q}(z)\tilde{N}^{\Omega}_{\gamma, q}(z) & \mbox{ in } B_{\rho}^+\\
\gamma_0\nabla R(z)\cdot \nu=(\gamma_0-\tilde{\gamma}(z))\nabla \tilde{N}^{\Omega}_{\gamma, q}(z)\cdot \nu\ ,
& \mbox{ on }B_{\rho}\cap \Pi_n,
\end{cases}
\end{equation}
The following pointwise estimates for the Neumann kernel and its gradient follow from 
\cite[Theorem~5.15, Corollary~5.16]{kim2024neumann} and from standard elliptic regularity results:
\begin{equation}\label{N}
    |{N}^{\Omega}_{\gamma,q}(x,0)|\le C |x|^{-1} \ \ \mbox{for any } \ x\in \Omega,
\end{equation}
 and
\begin{equation}\label{gN}
|\nabla {N}^{\Omega}_{\gamma,q}(x,0)|\le C |x|^{-2} \ \ \mbox{for any } \ x\in B_{\rho}\cap \Omega.
\end{equation}
As a consequence, we have that 
\begin{equation}\label{reminder}
|R(z)|+|z||\nabla_z R(z)|\leq C,\quad \text{ for every } z\in \partial B_{\rho}\cap \mathbb{R}^3_+ 
\end{equation}
Let us now consider $w\in B^+_{\rho}$. By Green's identity, we get 
\begin{eqnarray*}
R(w)&=&- \int_{B^+_{\rho}}R(z)\dive_z(\gamma_0\nabla_z N_{\gamma_0}(z,w))\mbox{ d}z+\nonumber\\
&=&\int_{B^+_{\rho}} (\gamma_0 - \tilde{\gamma}(z))\nabla_z \tilde{N}_{\gamma,q}^{\Omega}(z)\cdot \nabla_z N_{\gamma_0}(z,w) \mbox{ d}z - \int_{B^+_{\rho}} \tilde{q}(z) \tilde{N}_{\gamma,q}^{\Omega}(z)N_{\gamma_0}(z,w)\mbox{ d}z\nonumber\\
&&+\int_{\partial B_{\rho}^+} N_{\gamma_0}(z,w)(\tilde{\gamma}(z) - \gamma_0)\nabla_z \tilde{N}_{\gamma,q}^{\Omega}(z)\cdot \nu  \mbox{ d}\sigma +
  \int_{\partial B_{\rho}^+} N_{\gamma_0}(z,w) \gamma_0\nabla_z R(z)\cdot\nu \mbox{ d}\sigma+\nonumber\\
&& - \int_{\partial B^+_{\rho}}R(z)\gamma_0\nabla_z N_{\gamma_0}(z,w)\cdot\nu  \mbox{ d}\sigma
\end{eqnarray*}
If we split $\partial B^+_{\rho} = (\partial B_{\rho} \cap \mathbb{R}^n_+) \cup (B_{\rho}\cap \Pi_n )$ then, thank to \eqref{eqR} and  by the boundary condition in \eqref{NeumannKernelsemispazioint}, we get 
\begin{eqnarray}\label{R}
R(w)&=&\int_{B^+_{\rho}} (\gamma_0 - \tilde{\gamma}(z))\nabla_z \tilde{N}_{\gamma,q}^{\Omega}(z)\cdot \nabla_z N_{\gamma_0}(z,w) \mbox{ d}z 
- \int_{B_{\rho}^+} \tilde{q}(z) \tilde{N}_{\gamma,q}^{\Omega}(z)N_{\gamma_0}(z,w)\mbox{ d}z\\
&&\nonumber + \int_{\partial B_{\rho} \cap \mathbb{R}^n_+} N_{\gamma_0}(z,w)(\tilde{\gamma}(z)-\gamma_0)\nabla_z \tilde{N}_{\gamma,q}^{\Omega}(z) \cdot \nu \mbox{ d}\sigma + 
\int_{\partial B_{\rho} \cap \mathbb{R}^n_+} N_{\gamma_0}(z,w)\gamma_0\nabla_z R(z)\cdot\nu \mbox{ d}\sigma\\
&&\nonumber - \int_{\partial B_{\rho}\cap \mathbb{R}^n_+}R(z){\gamma_0}\nabla_z N_{\gamma_0}(z,w)\cdot \nu  \mbox{ d}\sigma.
\end{eqnarray}
By choosing $|w|<\frac{\rho}{2}$, we have that all the boundary integrals in \eqref{R} are uniformly bounded. This follows from \eqref{reminder} and from the fact that, if $|w|<\rho/2$, the singularity of $N_{\gamma_0}(\cdot,w)$ is inside $B_{\rho/2}$, so every boundary integral on $\partial B_\rho$ is an integral of a bounded integrand. 

Let us now estimate the volume integrals.
By \eqref{sigmalocale} and \eqref{gN} we have that 
\begin{equation}\label{gradienti}
    \left|\int_{B_{\rho}^+} (\gamma_0-\tilde{\gamma}(z))\nabla_z \tilde{N}_{\sigma,q}^{\Omega}(z)\cdot \nabla_z N_{\gamma_0}(z,w)\mbox{d}z \right|  \le C \int_{B_{\rho}^+}\frac{|z'|}{ |z|^{2}|z-w|^{2}}\mbox{d}z \le C \int_{B_{\rho}^+} |z|^{-1}|z-w|^{-2}\mbox{d}z.
\end{equation}
Let us set
\begin{equation}\label{raggioR}
    R:=\frac{|w|}{2},
\end{equation}
and notice that $R< \frac{\rho}{4}$.
Let us  write
\[B^+_\rho=\mathcal{F}_1\cup\mathcal{F}_2\cup\mathcal{F}_3\]
where
\begin{eqnarray*}
    &\mathcal{F}_1=\left\{z\in B^+_\rho\, :\, |z-w|\leq R\right\},\quad
    \mathcal{F}_2=\left\{z\in B^+_\rho\, :\, |z-w|>R \mbox{ and } |z|\leq 4R\right\}\\
    &\mathcal{F}_3=\left\{z\in B^+_\rho\, :\, |z-w|>R \mbox{ and } |z|> 4R\right\}
\end{eqnarray*}
In $ \mathcal{F}_1$ we have
\[|z|\geq |w|-|w-z|\geq 2R-R=R,\]
hence
\begin{equation}\label{F1}
\int_{\mathcal{F}_1}|z|^{-1}|z-w|^{-2}\mbox{d}z\leq \frac{1}{R}\int_{|z-w|\leq R}|z-w|^{-2}\mbox{d}z\leq 4\pi.
\end{equation}
In $\mathcal{F}_2$, we have
\[|z|^{-1}|z-w|^{-2}\leq R^{-2},\]
hence
\begin{equation}\label{F2}
\int_{\mathcal{F}_2}|z|^{-1}|z-w|^{-2}\mbox{d}z\leq \frac{1}{2R^2}\int_{|z|\leq 4R}|z|^{-1}\mbox{d}z= 16\pi.
\end{equation}
In $\mathcal{F}_3$, $|z-w|\geq |z|-|w|=|z|-2R\geq |z|/2$, and
\begin{equation}\label{F3}
\int_{\mathcal{F}_3}|z|^{-1}|z-w|^{-2}\mbox{d}z\leq 2\int_{4R<|z|<\rho}|z|^{-3}\mbox{d}z\leq 18\pi \ln\left(\frac{\rho}{4R}\right).
\end{equation}
Finally, from \eqref{gradienti}, \eqref{raggioR}, \eqref{F1}, \eqref{F2} and \eqref{F3}, we get
\begin{equation*}
    \left|\int_{B_{\rho}^+} (\gamma_0-\tilde{\gamma}(z))\nabla_z \tilde{N}_{\sigma,q}^{\Omega}(z)\cdot \nabla_z N_{\gamma_0}(z,w)\mbox{d}z \right|  \le C\left(1+\ln\left(\frac{\rho}{|w|}\right)\right)
\end{equation*}
where $C$ does not depend on $w$.
In a similar way we can get, from \eqref{N} and boundedness of $q$ (\eqref{hpalpha} and \eqref{linfty}), that 
\begin{equation*}
\left|\int_{B_{\rho}^+} \tilde{q}(z) \tilde{N}_{\gamma,q}^{\Omega}(z) N_{\gamma_0}(z,w)\mbox{d}z \right|  \le C |w| \ .
\end{equation*}
Hence we finally deduce that  $|R(z)|\le C \left|\ln|z|\right|$
on $B_{\rho}^+$ and, recalling that $|z|=\mathcal{O}(|x|)$, that we put $y'=0$, and \eqref{Nsemip}, the thesis follows. 
\end{proof}

\begin{lemma}\label{Lemma 4.4}
Let $y'\in\partial\Omega$ and suppose there exists a neighborhood $\mathcal{U}$ of $y'$
such that $\partial\Omega\cap \mathcal{U}$ is $C^{1,1}$ and  $\gamma=\gamma_0$ in $\mathcal{U}\cap\Omega$. Then the
knowledge of the boundary Neumann kernel
$N^\Omega_{\gamma,q}(x,y')$ for every $x\in\partial\Omega\cap \mathcal{U}$ uniquely
determines the tangential metric \[
\tilde{\mathfrak{g}}=\mathfrak{g}_0e_i\cdot e_j,\,\,i,j=1,2\] where $e_1, e_2$ are the directions of the tangent plane, $T_{y'}(\partial\Omega)$, at the point $y'$. 
\end{lemma}

\begin{proof}
Fix boundary normal coordinates so that $y'=0$, the tangent plane at $0$ is
$\Pi_3=\{x_3=0\}$ and the boundary near $0$ is given by $x_3=\varphi(x')$ with
$\varphi(0)=0$, $\nabla\varphi(0)=0$. By the local expansion of the Neumann
kernel derived in Proposition~4.3 we have, for $x$ near $0$,

\[
N^\Omega_{\gamma,q}(x,0)
=\frac{1}{2\pi}\Bigl(\mathfrak{g}_0x\cdot x\Bigr)^{-\frac{1}{2}}
+o\bigl(|\ln|x||\bigr).
\]

Writing \(x=(x',\varphi(x'))\) we therefore have
\begin{equation}\label{expN}
N^\Omega_{\gamma,q}\bigl((x',\varphi(x')),0\bigr)
=\frac{1}{2\pi}\bigl(\mathfrak{g}_0(x',\varphi(x'))\cdot(x',\varphi(x'))\bigr)^{-1/2}
+o\bigl(|\ln|x'||\bigr).
\end{equation}
Fix a unit tangential direction \(\xi'\in\mathbb R^2\) and put
\[
 x(r) := \bigl(r\xi',\varphi(r\xi')\bigr)\in\partial\Omega,\qquad r>0. 
\]
Because \(\varphi\in C^{1, 1}\) and \(\nabla\varphi(0)=0\) we have
\begin{equation}\label{opic}
\varphi(r\xi')=o(r)\quad\text{as }r\to0.
\end{equation}
We write the metric $\mathfrak{g}_0$  as a $3\times3$ symmetric positive matrix and represent  it as
\[
\mathfrak{g}_0=
\begin{pmatrix} \tilde{\mathfrak{g}} & b\\ b^T & (\mathfrak{g}_0)_{33}\end{pmatrix},
\qquad \tilde{\mathfrak{g}}\in\mathbb R^{2\times2},\ b\in\mathbb R^{2},\ (\mathfrak{g}_0)_{33}\in\mathbb R.
\]
Then for \(x(r)=(r\xi',\varphi(r\xi'))\),
\[
\mathfrak{g}_0x(r)\cdot x(r)
= r^2\tilde{\mathfrak{g}}\xi'\cdot\xi' + 2 r\xi'\cdot b\,\varphi(r\xi') +( \mathfrak{g}_0)_{33}\bigl(\varphi(r\xi')\bigr)^2\\
\]
Using  \eqref{opic} the last two terms are \(o(r^2)\), hence
\[
\mathfrak{g}_0x(r)\cdot x(r) =r^2\tilde{\mathfrak{g}}\xi'\cdot\xi'(1+o(1))
\quad\text{as }r\to 0. 
\]
By \eqref{expN} we have  
\[
\begin{aligned}
N^\Omega_{\gamma,q}\bigl(x(r),0\bigr)
&= \frac{1}{2\pi}r^{-1}\bigl(\tilde{\mathfrak{g}}\xi'\cdot\xi'\bigr)^{-1/2}\,(1+o(1)) + o(|\ln|r|).
\end{aligned}
\]
Multiply both sides by $r$ and letting $r$ go to zero we
\[
\lim_{r\to0} r\,N^\Omega_{\gamma,q}\bigl((r\xi',\varphi(r\xi')),0\bigr)
= \frac{1}{2\pi}\bigl(\tilde{\mathfrak{g}}\xi'\cdot\xi')\bigr)^{-1/2}.
\]
 Varying $\xi'$ over a basis of the tangent plane, that means knowing  $N^\Omega_{\gamma,q}(x,y')$ for $x$ in a neighborhood of $y'$ on the boundary, we recover all entries of the $2\times2$ matrix $\tilde{\mathfrak{g}}$.
\end{proof}

\begin{proposition}\label{lemma4.5}
Let $\Sigma\subset\partial\Omega$ be
an open nonflat portion of $\partial\Omega$ of class $C^{1,1}$, and let $\gamma=\gamma_0$ in a neighbourhood of $\Sigma$.
Then the knowledge of the Neumann kernel
$N^\Omega_{\gamma,q}(x',y')$ for all $x',y'\in\Sigma$ uniquely determines
$\gamma_0$.
\end{proposition}

\begin{proof}
We work in local boundary normal coordinates centered such that $0\in\Sigma$ , such that, in a neighborhood of $0$,
\begin{equation}
    \Omega \cap B_{r_0} =\{x\in B_{r_0} | x_3 > \varphi(x_1,x_2)  \}
\end{equation}
where $\varphi$ is given in Definition \ref{regbordo}.
By hypothesis $\gamma\equiv\gamma_0$ in a neighbourhood of $\Sigma$,
hence the associated metric
\[
\mathfrak{g}_0= (\det\gamma_0)\gamma_0^{-1}
\]
is constant on the patch $\Sigma$. Writing its entries in coordinates,
\[
\mathfrak{g}_0=\begin{pmatrix}
(\mathfrak{g}_0)_{11} & (g_0)_{12} & (\mathfrak{g}_0)_{13}\\
(\mathfrak{g}_0)_{12} & (g_0)_{22} & (\mathfrak{g}_0)_{23}\\
(\mathfrak{g}_0)_{13} & (\mathfrak{g}_0)_{23} & (\mathfrak{g}_0)_{33}
\end{pmatrix},
\]
Our goal is to recover the six unknown elements from boundary data.
Lemma~\ref{Lemma 4.4} implies that for any boundary point $P\in\Sigma$ and any orthonormal basis
$\{v_1^P,v_2^P\}$ of the tangent plane $T_P(\Sigma)$, the $2\times2$ matrix
\[
\tilde{ \mathfrak{g}}^P:=\bigl(\mathfrak{g}_0v_i^P\cdot v_j^P\bigr)_{i,j=1,2}
\]
is determined by the boundary Neumann kernel $N^\Omega_{\gamma,q}(\cdot,\cdot)$.
Let $\nu(P)=(\nu_1(P),\nu_2(P),\nu_3(P))$ be the outward unit normal to $\partial\Omega$
at a boundary point $P$, with $\nu_3(P)\neq0$.

We define the following parameters at $P$ by
\[
\alpha(P):=\partial_2\varphi(P)=-\frac{\nu_2(P)}{\nu_3(P)},\qquad
\delta(P):=\partial_1\varphi(P)=-\frac{\nu_1(P)}{\nu_3(P)}.
\]

Since $\Sigma$ is nonflat near $0$, the function $\varphi$ is not identically zero in
any neighbourhood of $0$ contained in $\Sigma$. As $\nabla\varphi$ is continuous,
it follows that at least one of $\partial_1\varphi$ or $\partial_2\varphi$ is not
identically zero near $0$. Therefore we can choose:
\begin{itemize}
\item[] $P_1=0$;
\item[] a point $P_2\in\Sigma$ arbitrarily close to $0$ such that
$\alpha(P_2)=\partial_2\varphi(P_2)\neq0$ and $\delta(P_2)=0$;
\item[] a point $P_3\in\Sigma$ arbitrarily close to $0$ such that
$\delta(P_3)=\partial_1\varphi(P_3)\neq0$..
\end{itemize}
This guarantees that the divisions by $\alpha$ and $\delta$ appearing below are allowed. Hence, without loss of generality, we can pick up 
\[
\nu(P_1)=-e_3,\qquad
\nu(P_2)=\frac{-e_3+\alpha e_2}{\sqrt{1+\alpha^2}},\qquad
\nu(P_3)=\frac{-e_3+\beta e_2+\delta e_1}{\sqrt{1+\beta^2+\delta^2}}.
\]
where $P_1=0$ and $\alpha:=\alpha(P_2)$, $\beta:=\alpha(P_3)=\partial_{x_2}\varphi(P_3)=-\nu_2(P_3)/\nu_3(P_3)$, and $\delta:=\delta(P_3)$. With the above choices of $P_1,P_2,P_3$, convenient orthonormal tangent bases are
\[
\begin{aligned}
&\text{at }P_1:\quad v_1^1=e_1,\quad v_2^1=e_2,\\[3pt]
&\text{at }P_2:\quad v_1^2=e_1,\quad
v_2^2=\frac{e_2+\alpha e_3}{\sqrt{1+\alpha^2}},\\[5pt]
&\text{at }P_3:\quad
v_1^3=\frac{e_1+\delta e_3}{\sqrt{1+\delta^2}},\quad
v_2^3=\frac{e_2+\beta e_3}{\sqrt{1+\beta^2}}.
\end{aligned}
\]

For each $k=1,2,3$ the measured data satisfy
\[
V_k^T \mathfrak{g}_0 V_k=\tilde{ \mathfrak{g}}^{P_k},\qquad V_k=[\,v_1^k\ v_2^k\,].
\]
We now solve these equations explicitly.

\smallskip\noindent
From $P_1$ we read directly
\[
\begin{pmatrix} (\mathfrak{g}_0)_{11} & (\mathfrak{g}_0)_{12}\\ (\mathfrak{g}_0)_{12} & (\mathfrak{g}_0)_{22}\end{pmatrix}=\tilde{\mathfrak{g}}^{P_1}.
\]

\smallskip\noindent
From the off–diagonal entry of $\tilde{\mathfrak{g}}^{P_2}$,
\[
\tilde{\mathfrak{g}}^{P_2}_{12}
=\frac{(\mathfrak{g}_0)_{12}+\alpha (\mathfrak{g}_0)_{13}}{\sqrt{1+\alpha^2}},
\]
we recover
\[
(\mathfrak{g}_0)_{13}=\frac{\sqrt{1+\alpha^2}\,\tilde{\mathfrak{g}}^{P_2}_{12}-(\mathfrak{g}_0)_{12}}{\alpha}.
\]

\smallskip\noindent
From the diagonal entry $\tilde{\mathfrak{g}}^{P_2}_{22}$ we obtain
\begin{equation}
    \label{L1}
2\alpha (\mathfrak{g}_0)_{23}+\alpha^2 (\mathfrak{g}_0)_{33}
=(1+\alpha^2)\tilde{\mathfrak{g}}^{P_2}_{22}-(\mathfrak{g}_0)_{22}.
\end{equation}

\smallskip\noindent
From the diagonal entry $\tilde{\mathfrak{g}}^{P_3}_{11}$,
\[
\tilde{\mathfrak{g}}^{P_3}_{11}
=\frac{(\mathfrak{g}_0)_{11}+2\delta (\mathfrak{g}_0)_{13}+\delta^2 (\mathfrak{g}_0)_{33}}{1+\delta^2},
\]
we solve
\[
(\mathfrak{g}_0)_{33}
=\frac{(1+\delta^2)
\tilde{\mathfrak{g}}^{P_3}_{11}-(\mathfrak{g}_0)_{11}-2\delta (\mathfrak{g}_0)_{13}}{\delta^2}.
\]

\smallskip\noindent
Substituting $(\mathfrak{g}_0)_{33}$ into \eqref{L1} yields
\[
(\mathfrak{g}_0)_{23}
=\frac{(1+\alpha^2)\tilde{\mathfrak{g}}^{P_2}_{22}-(\mathfrak{g}_0)_{22}-\alpha^2 (\mathfrak{g}_0)_{33}}{2\alpha}.
\]
Thus all six entries of $\mathfrak{g}_0$ are uniquely determined.

For $n=3$ the relation $\mathfrak{g}_0=(\det\gamma_0)\gamma_0^{-1}$ implies
$\det \mathfrak{g}_0=(\det\gamma_0)^2$, hence $\det\gamma_0=(\det \mathfrak{g}_0)^{1/2}$ and therefore
\[
\gamma_0=(\det \mathfrak{g}_0)^{1/2}\,\mathfrak{g}_0^{-1}.
\]
This completes the proof.
\end{proof}

\subsection{Proof of Theorem 4.1}
\begin{proof}[Proof of Theorem \ref{unicis}] 
Let $\gamma^{(1)}$ and $\gamma^{(2)}$, satisfying assumptions \eqref{gamma} and \eqref{ell}, such that 
\begin{equation*}
    \mathcal{N}^{\gamma^{(1)},\alpha}_{NL} (g)=\mathcal{N}^{\gamma^{(2)},\alpha}_{NL} (g)\mbox{ for every }g\in H^{-1/2}(\partial\Omega),
\end{equation*}
for a given $\alpha$ satisfying \eqref{hpalpha}.

By Theorem \ref{teoder} and Corollary \ref{corder}, we have that
\begin{equation}\label{mappelinug}
    \mathcal{N}^{\gamma^{(1)},q^{(1)}}_{L} (g)=\mathcal{N}^{\gamma^{(2)},q^{(2)}}_{L} (g)\mbox{ for every }g\in H^{-1/2}(\partial\Omega),
\end{equation}
where $q^{(i)}=3\alpha \left(u_0^{(i)}\right)^2$ and 
$u_0^{(i)}$, for $i=1,2$ is the solution to 
\begin{displaymath}
\left\{ \begin{array}{ll}
\mbox{div}\left(\gamma^{(i)}\nabla u_0^{(i)}\right) - \alpha \left(u_0^{(i)}\right)^3 =0, & \textnormal{in}\quad\Omega\\
\gamma^{(i)} \nabla u_0^{(i)}\cdot \nu=g_0,
& \textnormal{on}\quad{\partial\Omega},
\end{array} \right.
\end{displaymath}
By Proposition \ref{lemma4.5}, \eqref{mappelinug} implies that 
\begin{equation*}
    \gamma^{(1)}=\gamma^{(2)}=\gamma_0\mbox{ in }\Omega\setminus D.
\end{equation*}
Hence, we can write
\begin{equation*}
\gamma^{(i)}(x)=\gamma_0\chi_{\Omega\setminus D}+\gamma_1^{(i)}\chi_{D}\mbox{ for }i=1,2.
\end{equation*}

Moreover, being ${\mathcal{N}}_{NL}^{\gamma^{(1)}, \alpha}(g_0)={\mathcal{N}}_{NL}^{\gamma^{(2)}, \alpha}(g_0)$ we have that $u_0^{(1)}=u_0^{(2)}$ on $\partial \Omega$, so that
  $w=u_0^{(1)}-u_0^{(2)}$, solves the following Cauchy problem: 
\begin{displaymath}\label{CP}
\left\{ \begin{array}{ll}
\mbox{div}({\gamma_0}\nabla w) - \alpha(x) p(x)w =0, & \textnormal{in}\quad\Omega\setminus D \\
{\gamma_0} \nabla w\cdot \nu=0,
& \textnormal{on}\quad{\partial\Omega},\\
w=0 & \textnormal{on}\quad{\partial\Omega}
\end{array} \right.
\end{displaymath}
where $p(x)=\left(u_0^{(1)}\right)^2 +u_0^{(1)}u_0^{(2)} + \left(u_0^{(2)}\right)^2\in L^{\infty}(\Omega\setminus D)$, and $p(x)>0$ by Proposition \ref{defc0}.

By Proposition \ref{bound}  and the a priori assumptions on $\alpha$ we have that 
\[
\|\alpha p\|_{L^{\infty}(\Omega \setminus D)}\leq C
\]
where $C$ depends only on $\|g_0\|_{W_p^{1-\frac{1}{p}}(\partial\Omega)}$ and on $\alpha_0^{-1}$. Hence, by the uniqueness for the Cauchy problem together with the weak unique continuation property (see for example \cite{alessandrini2009stability}) it follows that  $w=0$ in $\Omega \setminus D$. 
Namely we have that $u_0^{(1)}=u_0^{(2)}$ in $\Omega \setminus D$ which in turn implies that $q^{(1)}=3\alpha \left(u_0^{(1)}\right)^2=3\alpha \left(u_0^{(2)}\right)^2=q^{(2)}$ in $\Omega \setminus D$.

We now need to show that $\gamma^{(1)}_1=\gamma_1^{(2)}$.
In order to do this, we show that \eqref{mappelinug} and the fact that 
\[\gamma^{(1)}=\gamma^{(2)} \textrm{ and } q^{(1)}=q^{(2)} \text{ in } \Omega\setminus D\] imply that
\begin{equation*}
    {\mathcal{N}}_{L, \partial D}^{\gamma^{(1)}, q^{(1)}}={\mathcal{N}}_{L, \partial D}^{\gamma^{(2)}, q^{(2)}}
\end{equation*}
where we denote by 
${\mathcal{N}}_{L, \partial D}^{\gamma^{(i)}, q^{(i)}}$ the Neumann to Dirichlet map for the domain $D$ relative to the conductivity and potential $\gamma^{(i)}, q^{(i)}$, $i=1,2$. This part follows the lines of the proof in \cite{Al-dH-G} (and of the preprint \cite{donlon2025uniqueness}).

We denote by 
\begin{equation*}
\Omega_0 =\{x\in \mathbb{R}^n : \mbox{dist}(x,\overline{\Omega})< r_0 \}  \ .
\end{equation*}
For  $i=1,2$, we 
extend $\gamma^{(i)}$ and $q^{(i)}$ to $\Omega_0\setminus\Omega$ by setting 
$\gamma^{(i)}=\gamma_0$ and $q^{(i)}=1$. For sake of shortness we still denote by $\gamma^{(i)}$ and $q^{(i)}$ these extensions.

We consider the Neumann kernel ${\tilde{N}}^{\Omega_0}_{\gamma^{(i)}, q^{(i)}}$ that solves 
\begin{displaymath}\label{neumann kernel enlarged domain}
\left\{ \begin{array}{ll}
\mbox{div}(\gamma^{(i)}(\cdot)\nabla {\tilde{N}}^{\Omega_0}_{\gamma^{(i)}, q^{(i)}}(\cdot,y)) - q^{(i)}(\cdot){\tilde{N}}^{\Omega_0}_{\gamma^{(i)}, q^{(i)}}(\cdot,y) =-\delta(\cdot -y), & \textnormal{in}\quad\Omega_0\\
{\gamma}^i (\cdot)\nabla \tilde{N}^{\Omega_0}_{\gamma^{(i)}, q^{(i)}}(\cdot, y)\nu =0 \ ,
& \textnormal{on}\quad{\partial\Omega_0},
\end{array} \right.
\end{displaymath}

Given $\psi \in C^{0,1}(\partial D)$, let $u^{(i)}, \ i=1,2$ be the solution to 

\begin{equation*}
\left\{ \begin{array}{ll}
\mbox{div}({\gamma^{(i)}}\nabla u^{(i)}) - {q^{(i)}} u^{(i)} =0, & \textnormal{in}\quad D \ ,\\
\gamma^{(i)} \nabla u^{(i)}\cdot \nu =\psi \ ,
& \textnormal{on}\quad{\partial D},
\end{array} \right.
\end{equation*}
We now consider a bounded extension operator 
\begin{eqnarray*}
 T: H^{\frac{1}{2}}(\partial D) \rightarrow H^1(\Omega)    
\end{eqnarray*}
such that, given $f\in H^{\frac{1}{2}}(\partial D)$ we have 
\begin{eqnarray*}
    T f|_{\partial \Omega}=0 \ \ .
\end{eqnarray*}
We set 

\begin{displaymath}
\bar{u}_i = \left\{ \begin{array}{ll}
u^{(i)}, & \textnormal{in}\quad D \ ,\\
T\left(u^{(i)}|_{\partial D} \right),
& \textnormal{in}\quad \Omega \setminus D ,
\end{array} \right.
\end{displaymath}

We have that $\bar{u}_i\in H^1(\Omega)$. Hence, for $x\in D$, it follows
\begin{eqnarray*}
 u^{(i)}(x)&=&\int_{\Omega}\bar{u}_i (y)\left (-\mbox{div}_y( {\gamma^{(i)}}(y)\nabla_y {\tilde{N}}^{\Omega_0}_{\gamma^{(i)}, q^{(i)}}(y,x)) ) + q^{(i)}(y){\tilde{N}}^{\Omega_0}_{\gamma^{(i)}, q^{(i)}}(y,x) \right)\mbox{d}y \\
 & =& \int_{\Omega} \gamma^{(i)}(y) \nabla_y \bar{u}_i (y) \cdot \nabla_y {\tilde{N}}^{\Omega_0}_{\gamma^{(i)}, q^{(i)}}(y,x) + q^{(i)}(y) \bar{u}_i(y){\tilde{N}}^{\Omega_0}_{\gamma^{(i)}, q^{(i)}}(y,x) \mbox{d}y +  \\
 &&- \int_{\partial \Omega} \bar{u}_i (y) \gamma ^i(y) \nabla {\tilde{N}}^{\Omega_0}_{\gamma^{(i)}, q^{(i)}}(y,x) \cdot \nu(y)\mbox{d}\sigma (y)  
\end{eqnarray*}

We observe that 

\begin{equation}\label{homdir}
  \bar{u}_i =0 \ \ \mbox{on}\ \partial \Omega 
\end{equation}
hence we have 
\begin{eqnarray*}
\int_{\partial \Omega}  \bar{u}_i(y) \gamma_i(y)\nabla {\tilde{N}}^{\Omega_0}_{\gamma^{(i)}, q^{(i)}}(y,x)\cdot \nu(y)\mbox{d}\sigma (y)  =0
 \end{eqnarray*}
Using the fact that $\bar u_i$ satisfies the same equation as $u^{(i)}$ in $D$ and \eqref{homdir}, we find that 
\begin{eqnarray*}
&&u^{(i)}(x)=\int_{\Omega} \gamma^{(i)}(y) \nabla_y \bar{u}_i (y) \cdot \nabla_y {\tilde{N}}^{\Omega_0}_{\gamma^{(i)}, q^{(i)}}(y,x) + q^{(i)}(y) \bar{u}_i(y){\tilde{N}}^{\Omega_0}_{\gamma^{(i)}, q^{(i)}}(y,x) \mbox{d}y =\\
&&\int _{D} \gamma^{(i)}(y) \nabla_y \bar{u}_i (y) \cdot \nabla_y {\tilde{N}}^{\Omega_0}_{\gamma^{(i)}, q^{(i)}}(y,x) + q^{(i)}(y) \bar{u}_i(y){\tilde{N}}^{\Omega_0}_{\gamma^{(i)}, q^{(i)}}(y,x) \mbox{d}y\\
&& + \int_{\Omega \setminus D} \gamma^{(i)}(y) \nabla_y \bar{u}_i (y) \cdot \nabla_y {\tilde{N}}^{\Omega_0}_{\gamma^{(i)}, q^{(i)}}(y,x) + q^{(i)}(y) \bar{u}_i(y){\tilde{N}}^{\Omega_0}_{\gamma^{(i)}, q^{(i)}}(y,x) \mbox{d}y=\\
&& \int_{\partial D} \psi(y) {\tilde{N}}^{\Omega_0}_{\gamma^{(i)}, q^{(i)}}(y,x) \mbox{d}\sigma(y) +\int_{\Omega \setminus D} \gamma^{(i)}(y) \nabla_y \bar{u}_i (y) \cdot \nabla_y {\tilde{N}}^{\Omega_0}_{\gamma^{(i)}, q^{(i)}}(y,x) + q^{(i)}(y) \bar{u}_i(y){\tilde{N}}^{\Omega_0}_{\gamma^{(i)}, q^{(i)}}(y,x) \mbox{d}y 
\end{eqnarray*}
which gives
\begin{eqnarray}\label{repformula}
u^{(i)}(x)&=&\int_{\partial D} \psi(y) {\tilde{N}}^{\Omega_0}_{\gamma^{(i)}, q^{(i)}}(y,x) \mbox{d}\sigma(y) +\\&&+\int_{\Omega\setminus D} \gamma^{(i)}(y) \nabla_y \bar{u}_i (y) \cdot \nabla_y {\tilde{N}}^{\Omega_0}_{\gamma^{(i)}, q^{(i)}}(y,x) + q^{(i)}(y) \bar{u}_i(y){\tilde{N}}^{\Omega_0}_{\gamma^{(i)}, q^{(i)}}(y,x) \mbox{d}y\nonumber
\end{eqnarray}

By \eqref{repformula}, differentiation under the integral signs and by using Fubini Theorem we have that, for $x\in D$,
\begin{eqnarray}\label{prodgrad}
&&(\gamma^{(1)} -\gamma^{(2)})(x)\nabla_x u^{(1)}(x)\cdot \nabla_x u^{(2)}(x)= \\
&&\int_{\partial D \times \partial D} \psi(y)\psi(z)(\gamma^{(1)} -\gamma^{(2)})(x)\nabla_x \tilde{N}^{\Omega_0}_{\gamma^{(1)},q^{(1)}}(y,x)\cdot \nabla_x \tilde{N}^{\Omega_0}_{\gamma^{(2)},q^{(2)}}(z,x)\mbox{d}\sigma(y)\mbox{d}\sigma(z)\mbox{d}y \ + \nonumber\\
&&\int_{\partial D \times (\Omega\setminus D)} \psi(y)\gamma^{(2)}_{lk}(z) \partial_{z_l}
\bar{u}_2(z) \partial_{z_k} [(\gamma^{(1)} -\gamma^{(2)})(x)\nabla_x \tilde{N}^{\Omega_0}_{\gamma^{(1)},q^{(1)}}(y,x)\cdot \nabla_x \tilde{N}^{\Omega_0}_{\gamma^{(2)},q^{(2)}}(z,x) ) ]\mbox{d}\sigma(z)\mbox{d}y + \nonumber\\
&& \int_{\partial D \times (\Omega\setminus D)} \psi(y) q^{(2)}(z) \bar{u}_2(z)  [(\gamma^{(1)} -\gamma^{(2)})(x)\nabla_x \tilde{N}^{\Omega_0}_{\gamma^{(1)},q^{(1)}}(y,x)\cdot \nabla_x \tilde{N}^{\Omega_0}_{\gamma^{(2)},q^{(2)}}(z,x) ) ]\mbox{d}\sigma(z)\mbox{d}y +\nonumber\\
&&\int_{(\Omega\setminus D) \times \partial D} \psi(z)\gamma^{(1)}_{lk}(y) \partial_{y_l}
\bar{u}_1(y) \partial_{y_k} [(\gamma^{(1)} -\gamma^{(2)})(x)\nabla_x \tilde{N}^{\Omega_0}_{\gamma^{(1)},q^{(1)}}(y,x)\cdot \nabla_x \tilde{N}^{\Omega_0}_{\gamma^{(2)},q^{(2)}}(z,x) ) ]\mbox{d}\sigma(z)\mbox{d}y + \nonumber\\
&& \int_{(\Omega\setminus D)\times \partial D} \psi(z) q^{(1)}(y) \bar{u}_1(y)  [(\gamma^{(1)} -\gamma^{(2)})(x)\nabla_x \tilde{N}^{\Omega_0}_{\gamma^{(1)},q^{(1)}}(y,x)\cdot \nabla_x \tilde{N}^{\Omega_0}_{\gamma^{(2)},q^{(2)}}(z,x) )]\mbox{d}\sigma(z)\mbox{d}y+\nonumber\\
&& \int_{(\Omega \setminus D) \times (\Omega \setminus D)} \bar{u}_1(y) \bar{u}_2(z) q^{(1)}(y)q^{(2)}(z) [(\gamma^{(1)} -\gamma^{(2)})(x)\nabla_x \tilde{N}^{\Omega_0}_{\gamma^{(1)},q^{(1)}}(y,x)\cdot \nabla_x \tilde{N}^{\Omega_0}_{\gamma^{(2)},q^{(2)}}(z,x) ) ]+ \nonumber\\
&&\int_{(\Omega \setminus D) \times (\Omega \setminus D)} \gamma^{(1)}_{lk}(y)q^{(2)}(z) \partial_{y_l} \bar{u}_1(y)  \bar{u}_2(z) \partial_{y_k}  [(\gamma^{(1)} -\gamma^{(2)})(x)\nabla_x \tilde{N}^{\Omega_0}_{\gamma^{(1)},q^{(1)}}(y,x)\cdot \nabla_x \tilde{N}^{\Omega_0}_{\gamma^{(2)},q^{(2)}}(z,x) )] \mbox{d}z\mbox{d}y+\nonumber\\
&& \int_{(\Omega \setminus D) \times (\Omega \setminus D)} \gamma^{(2)}_{lk}(z)q^{(1)}(y) \partial_{z_l} \bar{u}_2(z)  \bar{u}_1(y) \partial_{y_k}  [(\gamma^{(1)} -\gamma^{(2)})(x)\nabla_x \tilde{N}^{\Omega_0}_{\gamma^{(1)},q^{(1)}}(y,x)\cdot \nabla_x \tilde{N}^{\Omega_0}_{\gamma^{(2)},q^{(2)}}(z,x) )]\mbox{d}z\mbox{d}y +\nonumber\\
&& \int_{(\Omega \setminus D) \times (\Omega \setminus D)}  
\gamma^{(1)}_{lk}(z)\partial_{z_l} \bar{u}_2(z) \gamma^{(1)}_{mn}(y)\partial_{y_n} \bar{u}_1(y)\partial_{z_k}\partial_{y_m}  [(\gamma^{(1)} -\gamma^{(2)})(x)\nabla_x \tilde{N}^{\Omega_0}_{\gamma^{(1)},q^{(1)}}(y,x)\cdot \nabla_x \tilde{N}^{\Omega_0}_{\gamma^{(2)},q^{(2)}}(z,x) )\mbox{d}z\mbox{d}y \nonumber
\end{eqnarray}
Moreover, we have that, for $x\in D$,
\begin{eqnarray}\label{prodfunz}
  && (q^{(1)}-q^{(2)})(x)u^{(1)}(x)u^{(2)}(x)= \int_{\partial D \times \partial D} \psi(y)\psi(z)(q^{(1)}-q^{(2)})(x) \tilde{N}^{\Omega_0}_{\gamma^{(1)},q^{(1)}}(y,x)\cdot\tilde{N}^{\Omega_0}_{\gamma^{(2)},q^{(2)}}(z,x) \mbox{d}\sigma(y)\mbox{d}\sigma(z)+ \nonumber\\
   && \int_{\partial D \times (\Omega \setminus D)}  \psi(y)\gamma^{(2)}_{lk}\partial_{z_l}\bar{u}_2 \partial_{z_k} ((q^{(1)}-q^{(2)})(x) \tilde{N}^{\Omega_0}_{\gamma^{(1)},q^{(1)}}(y,x)\cdot\tilde{N}^{\Omega_0}_{\gamma^{(2)},q^{(2)}}(z,x) )\mbox{d}\sigma(y)\mbox{d}z +\\
   &&\int_{\partial D\times (\Omega\setminus D)} q_2(z)\bar{u}_2(z) (q^{(1)}-q^{(2)})(x) \tilde{N}^{\Omega_0}_{\gamma^{(1)},q^{(1)}}(y,x)\cdot\tilde{N}^{\Omega_0}_{\gamma^{(2)},q^{(2)}}(z,x)\mbox{d}\sigma(y)\mbox{d}z + \nonumber\\
   && \int_{(\Omega \setminus D)\times \partial D} \psi(z)\gamma^{(1)}_{lk}(y)\partial_{y_l} \bar{u}_1(y) \partial_{y_k} (q^{(1)}-q^{(2)})(x) \tilde{N}^{\Omega_0}_{\gamma^{(1)},q^{(1)}}(y,x)\cdot\tilde{N}^{\Omega_0}_{\gamma^{(2)},q^{(2)}}(z,x)\mbox{d}\sigma(z)\mbox{d}y +\nonumber\\
   && \int_{(\Omega \setminus D)\times \partial D} q^{(1)}(y)\bar{u}_1(y)(q^{(1)}-q^{(2)})(x) \tilde{N}^{\Omega_0}_{\gamma^{(1)},q^{(1)}}(y,x)\cdot\tilde{N}^{\Omega_0}_{\gamma^{(2)},q^{(2)}}(z,x)\mbox{d}\sigma(z)\mbox{d}y +\nonumber\\
   && \int_{(\Omega\setminus D)\times (\Omega \setminus D)} \bar{u}_1(y)\bar{u}_2(z)q^{(1)}(y)q^{(2)}(z)(q^{(1)}-q^{(2)})(x) \tilde{N}^{\Omega_0}_{\gamma^{(1)},q^{(1)}}(y,x)\cdot\tilde{N}^{\Omega_0}_{\gamma^{(2)},q^{(2)}}(z,x) \mbox{d}z\mbox{d}y +\nonumber\\
   && \int_{(\Omega\setminus D)\times (\Omega \setminus D)} 
\gamma^{(1)}_{lk}(y)q^{(2)}(z)\partial_{y_l}\bar{u}_1(y)\bar{u}_2(z)\partial_{y_k}((q^{(1)}-q^{(2)})(x) \tilde{N}^{\Omega_0}_{\gamma^{(1)},q^{(1)}}(y,x)\cdot\tilde{N}^{\Omega_0}_{\gamma^{(2)},q^{(2)}}(z,x) )\mbox{d}z\mbox{d}y +\nonumber\\
&& \int_{(\Omega\setminus D)\times (\Omega \setminus D)} 
\gamma^{(2)}_{lk}(z)q^{(1)}(y)\bar{u}_1(y)\partial_{z_l}\bar{u}_2(z)\partial_{z_k}((q^{(1)}-q^{(2)})(x) \tilde{N}^{\Omega_0}_{\gamma^{(1)},q^{(1)}}(y,x)\cdot\tilde{N}^{\Omega_0}_{\gamma^{(2)},q^{(2)}}(z,x) )\mbox{d}z\mbox{d}y +\nonumber\\
&&\int_{(\Omega\setminus D)\times (\Omega \setminus D)} \gamma^{(2)}_{lk}(z)\partial_{z_l}\bar{u}_2(z)\gamma^{(1)}_{nm}(y)\partial_{y_n}\bar{u}_1(y)\partial_{z_k}\partial_{y_l}(q^{(1)}-q^{(2)})(x) \tilde{N}^{\Omega_0}_{\gamma^{(1)},q^{(1)}}(y,x)\cdot\tilde{N}^{\Omega_0}_{\gamma^{(2)},q^{(2)}}(z,x) \mbox{d}z\mbox{d}y \nonumber
\end{eqnarray}

For any $y,z\in \Omega_0 \setminus D$ we define 

\begin{eqnarray}\label{S}
&&S(y,z) = \int_{D}(\gamma^{(1)}(x)-\gamma^{(2)}(x))\nabla_x \tilde{N}^{\Omega_0}_{\gamma^{(1)},q^{(1)}}(y,x)\cdot \nabla_x \tilde{N}^{\Omega_0}_{\gamma^{(2)},q^{(2)}}(z,x) )\mbox{d}x + \\
&&\int_{D} (q^{(1)}(x)-q^{(2)}(x))\tilde{N}^{\Omega_0}_{\gamma^{(1)},q^{(1)}}(y,x)\cdot \tilde{N}^{\Omega_0}_{\gamma^{(2)},q^{(2)}}(z,x) ) \mbox{d}x.\nonumber
\end{eqnarray}

For any $y,z \in \Omega_0 \setminus D$ we have that 
\begin{eqnarray*}
&&\mbox{div}_y(\gamma^{(1)}(y)\nabla_y S(y,z)) - q^{(1)}(y)S(y,z)=0\\
&&\mbox{div}_z(\gamma^{(2)}(z)\nabla_z S(y,z)) - q^{(2)}(z)S(y,z)=0 \ \ .
\end{eqnarray*}

Moreover, since 
\begin{eqnarray*}
 \gamma^{(1)}(x)=\gamma^{(2)}(x)\ \ x\in \ \Omega\setminus D ;\\   
  q^{(1)}(x)=q^{(2)}(x)\ \ x\in \ \Omega\setminus D \ \ .
\end{eqnarray*}
we have that

\begin{eqnarray}\label{Somega}
&&S(y,z)=\int_{\Omega}(\gamma^{(1)}(x)-\gamma^{(2)}(x))\nabla_x \tilde{N}^{\Omega_0}_{\gamma^{(1)},q^{(1)}}(y,x)\cdot \nabla_x \tilde{N}^{\Omega_0}_{\gamma^{(2)},q^{(2)}}(z,x) )\mbox{d}x + \\
&&\int_{\Omega} (q^{(1)}(x)-q^{(2)}(x))\tilde{N}^{\Omega_0}_{\gamma_1,q_1}(y,x) \tilde{N}^{\Omega_0}_{\gamma^{(2)},q^{(2)}}(z,x) ) \mbox{d}x\nonumber 
\end{eqnarray}

For $y,z\in 
S_0 =\{x\in \mathbb{R}^n : 0<\mbox{dist}(x,\overline{\Omega})< r_0 \} $, the functions   $\tilde{N}^{\Omega_0}_{\gamma^{(1)},q^{(1)}}(y,\cdot)$ and $\tilde{N}^{\Omega_0}_{\gamma^{(2)},q^{(2)}}(z,\cdot)$ solve 
in $\Omega$ the equations
\begin{equation*}
\mbox{div}_x\left(\gamma^{(1)}(x)\nabla_x \tilde{N}^{\Omega_0}_{\gamma^{(1)},q^{(1)}}(y,x)\right) - q^{(1)}(x)\tilde{N}^{\Omega_0}_{\gamma^{(1)},q^{(1)}}(y,x)=0\
\end{equation*}
and 
\begin{equation*}
\mbox{div}_x\left(\gamma^{(2)}(x)\nabla_x \tilde{N}^{\Omega_0}_{\gamma^{(2)},q^{(2)}}(z,x)\right) - q^{(2)}(x)\tilde{N}^{\Omega_0}_{\gamma^{(2)},q^{(2)}}(z,x)=0.
\end{equation*}
By Alessandrini identity \cite[Chap.5]{IsakovBook} we then have that, by \eqref{Somega},
\begin{eqnarray*}
S(y,z) = \langle \gamma^{(1)}(x) \nabla_x \tilde{N}^{\Omega_0}_{\gamma^{(1)},q^{(1)}}(y, \cdot) \cdot \nu, (\mathcal{N}_{L}^{\gamma^{(1)},q^{(1)}} - \mathcal{N}_{L}^{\gamma^{(2)},q^{(2)}})\gamma^{(2)}(x) \nabla_x \tilde{N}^{\Omega_0}_{\gamma^{(2)},q^{(2)}}(z, \cdot) \cdot \nu   
\rangle_{\partial\Omega}.
\end{eqnarray*}
Hence, by \eqref{mappelinug},
\begin{equation*}
    S(y,z)=0 \mbox{ for every }y,z\in \Omega_0\setminus\Omega.
\end{equation*}
Since $S(y,z)$ satisfies the unique continuation properties in each variable $y,z\in \Omega\setminus D$, by \cite{alessandrini2009stability},  we have that 
\begin{eqnarray*}
    S(y,z)=0 \ \ \ \mbox{for all}\ \ y,z\in \Omega\setminus D
\end{eqnarray*}

So, by \eqref{prodgrad}, \eqref{prodfunz}, \eqref{S}, and extending $S$ and $\nabla S$ to zero on $\partial D$
\begin{eqnarray}\label{magia}
&&\int_{D} [(\gamma^{(1)}(x)-\gamma^{(2)}(x)) \nabla_x u^{(1)}(x)\cdot \nabla_x u^{(2)}(x) + (q^{(1)}(x)-q^{(2)}(x))u^{(1)}(x)u^{(2)}(x)]\mbox{d}x =\nonumber\\
&& \int_{\partial D\times \partial D} \psi(y)\psi(z)S(y,z)\mbox{d}\sigma(y) \mbox{d}\sigma(z) +\nonumber\\
&& \int_{\partial D \times (\Omega\setminus D)} \psi(y) [q^{(2)}(z) \bar{u}_2(z)S(y,z) +\gamma^{(2)}_{lk}(z)\partial_{z_l}\bar{u}_2(z)\partial_{z_k}S(y,z) ] \mbox{d}\sigma(y) \mbox{d}z\nonumber\\
&&\int_{(\Omega\setminus D) \times \partial D} \psi(z)[q^{(1)}(y)\bar{u}_1(y)S(y,z) + \gamma^{(1)}_{lk}(y)\partial_{y_l}\bar{u}_1(y)\partial_{y_k}S(y,z)  ]\mbox{d}\sigma(z) \mbox{d}y \nonumber\\
&& \int_{(\Omega\setminus D)\times (\Omega\setminus D)} \gamma^{(2)}_{lk}(z)\partial_{z_l}\bar{u}_2(z)\gamma^{(1)}_{nm}(y)\partial_{y_n}\bar{u}_1(y)\partial_{z_k}\partial_{y_l}S(y,z)\mbox{d}y\mbox{d}z +\\
&& \int_{(\Omega\setminus D)\times (\Omega\setminus D)} \gamma^{(1)}_{lk}(y)q^{(2)}(z)\partial_{y_l}\bar{u}_1(y)\bar{u}_2(z)\partial_{y_k}S(y,z)\mbox{d}y\mbox{d}z + \nonumber\\
&&\int_{(\Omega \setminus D)\times (\Omega\setminus D)} \gamma^{(2)}_{lk}(z)q^{(1)}(y)\bar{u}_1(y)\partial_{z_l}\bar{u}_2(z)\partial_{z_k}S(y,z) \mbox{d}y\mbox{d}z +\nonumber\\
&&\int_{(\Omega \setminus D)\times (\Omega\setminus D)} \bar{u}_1(y)\bar{u}_2(z)q^{(1)}(y)q^{(2)}(z)S(y,z) \mbox{d}y\mbox{d}z =0\nonumber
\end{eqnarray}

By Alessandriniidentity applied on $D$  and by \eqref{magia}, we have for any $\psi$
\begin{eqnarray*}
 \langle \psi, (\mathcal{N}_{L,\partial D}^{\gamma^{(1)},q^{(1)}}-\mathcal{N}_{L,\partial D}^{\gamma^{(2)},q^{(2)}})\psi \rangle_{\partial D}   &=& \int_{D} [(\gamma^{(1)}(x)-\gamma^{(2)}(x)) \nabla_x u^{(1)}(x)\cdot \nabla_x u^{(2)}(x) \\&&+ (q^{(1)}(x)-q^{(2)}(x))u^{(1)}(x)u^{(2)}(x)]\mbox{d}x=0,
\end{eqnarray*}
and, since the linear NtoD map is a self-adjoint operator, we get $$\mathcal{N}_{L, \partial D}^{\gamma^{(1)},q^{(1)}}=\mathcal{N}_{L,\partial D}^{\gamma^{(2)},q^{(2)}}\ .$$

At this point we can use Proposition \ref{lemma4.5}, where the maps $\mathcal{N}_{L,\partial D}^{\gamma^{(i)},q^{(i)}} $
takes the place of $\mathcal{N}_{L}^{\gamma^{(i)},q^{(i)}} $  and the boundary $\partial D $ takes the place of $\partial \Omega$, to conclude that $\gamma_1^{(1)}=\gamma_1^{(2)}$. Therefore,  $\gamma^{(1)}=\gamma^{(2)}$ in $\Omega$ which concludes the proof.
\end{proof}
\subsection{Final remarks and extensions}
\begin{remark}
The uniqueness result of Theorem~\ref{unicis} still holds if one replaces the global
NtoD map with its \emph{local} counterpart on a nonempty, non-flat open portion $\Sigma\subset\partial\Omega$. Indeed, one can verify that all the results
in Sections~2 and~3 remain valid in this
setting. The results of Section 4 can be easily adapted to the local NtoD map, and hence 
the conclusions of Theorem~\ref{unicis} continue to hold.
\end{remark}
\begin{remark}
The identification of the inclusion $D$ from local Neumann-to-Dirichlet data
has also been studied in the linear setting by Alessandrini in \cite{Al-dH-G-S} and, for an anisotropic Schr\"odinger type equation in the following preprint \cite{donlon2025uniqueness}. Roughly speaking, their results
show that, under suitable geometric hypotheses (i.e. all the boundary of the inclusion $D$ is non flat) and with some visibility assumption (i.e. that the coefficient inside the inclusion is different from the one outside the inclusion), one can recover both the
piecewise constant anisotropic conductivity and the shape of the inclusion
from local measurements on $\Sigma$.

The proof is
purely by contradiction, and therefore non-constructive: it does not provide
an explicit procedure for reconstructing $D$, but only rules out the existence
of two different inclusions compatible with the same data. 

\end{remark}


\begin{remark}
From the above proof, it is clear that we can extend the uniqueness theorem thoroughly to the case of $N$ nested domains $\{\Omega_j\}_{j=0}^N$ where $\Omega_0:=\Omega$ and $\Omega_N\subset\!\subset\Omega_{N-1}\subset\!\subset\dots\Omega_1 \subset\!\subset\Omega_0$.  We expect a similar result to hold also in the case of more general partitions.
\end{remark}

Several questions remain open. We do not address stability, which is essential for quantitative reconstructions. The simultaneous identification of \(\gamma\) and \(\alpha\) from NtD data, particularly in the anisotropic case and
when linearizing around nontrivial backgrounds, is still unresolved. Extensions to inclusions touching the boundary, and to more general conductivity partitions, as well as the design of robust numerical algorithms
consistent with this analytical framework, are natural directions for future work.


\section{Conclusions}

We have studied an inverse boundary value problem for a stationary semilinear elliptic equation arising in cardiac electrophysiology, where anisotropic conduction is modeled by a symmetric tensor \(\gamma\) and nonlinear ionic effects
by a cubic term \(\alpha u^{3}\). The measurements are encoded by the nonlinear NtD map, corresponding to pacing currents prescribed on the boundary and the resulting stationary transmembrane potentials.

Under the assumption that \(\alpha\)  is known and that \(\gamma(x)=\gamma_0\chi_{\Omega\setminus D}(x)+\gamma_1\chi_{D}(x)\), we showed that the nonlinear NtD map uniquely determines the anisotropic conductivities $\gamma_0$ and $\gamma_1$. A key feature of our approach is that the linearization is carried out around a nontrivial strictly positive solution associated with a physiologically meaningful pacing current, rather
than around the trivial state. In contrast to previous works on inverse problems for semilinear PDEs, which are formulated in terms of the DtN map, our analysis is based on the NtD map, which better reflects the pacing-guided experimental setup. To our knowledge, this is the first uniqueness result for anisotropic conductivities from NtD data in such a nonlinear
setting.

\section*{Aknowledgments}
The authors want to thank Stefano Pagani of Politecnico of Milan for bringing this problem to their attention.\\
EB has been partially supported by NYUAD Science Program Project Fund AD364.
The work of EF and ES have been supported by Gruppo Nazionale per l’Analisi Matematica, la Probabilità e le loro applicazioni (GNAMPA) by the grant ”Problemi inversi per equazioni alle derivate parziali”. EF and ES are supported by the Italian MUR through the PRIN 2022 project “Inverse problems in PDE: theoretical and numerical analysis”, project code 2022B32J5C, under the National Recovery and Resilience Plan (PNRR), Italy, funded by the European Union- Next Generation EU, Mission 4 Component 1 CUP F53D23002710006.

\bibliographystyle{alpha}
\bibliography{sample}
\end{document}